\documentclass[a4paper]{amsart}
\usepackage[utf8]{inputenc}
\usepackage[T1]{fontenc}
\usepackage{amssymb}
\usepackage{lmodern}
\usepackage{microtype}
\usepackage{csquotes}
\usepackage{mathrsfs}
\usepackage{mathtools}
\usepackage{xspace, comment}
\usepackage[dvipsnames]{xcolor}
\usepackage{enumitem}
\usepackage[american]{babel}
\usepackage[all]{xy}
\usepackage[backref=page]{hyperref}
\usepackage[nameinlink, capitalize]{cleveref}
\hypersetup{
pdfauthor={Patrick Graf, Tim Kirschner},
pdftitle={Finite quotients of three-dimensional complex tori},
pdfkeywords={Complex tori, torus quotients, vanishing Chern classes, second orbifold Chern class, Minimal Model Program, klt singularities},
pdfstartview={Fit},
pdfpagelayout={OneColumn},
pdfpagemode={UseNone},
linktoc=all,
breaklinks,
linkcolor=[RGB]{0 0 96},
citecolor=[RGB]{96 0 0},
urlcolor=[RGB]{0 96 0},
colorlinks
}

\setlist[enumerate]{
label=(\thetheo.\arabic*),
before={\setcounter{enumi}{\value{equation}}},
after={\setcounter{equation}{\value{enumi}}},
itemsep=1ex
}
\setlist[itemize]{
leftmargin=*,
itemsep=1ex,
label=$\circ$
}
\theoremstyle{plain}
\newtheorem{theo}{Theorem}[section]
\newtheorem{lemm}[theo]{Lemma}
\newtheorem{coro}[theo]{Corollary}
\newtheorem{prop}[theo]{Proposition}
\newtheorem{conj}[theo]{Conjecture}
\theoremstyle{definition}
\newtheorem{defi}[theo]{Definition}
\newtheorem{propdefi}[theo]{Proposition/Definition}
\newtheorem{nota}[theo]{Notation}

\theoremstyle{remark}
\newtheorem{rema}[theo]{Remark}
\newtheorem{exam}[theo]{Example}
\newtheorem{clai}[theo]{Claim}
\newtheorem{awlog}[theo]{Additional Assumption}
\newtheorem*{rema*}{Remark}
\crefname{rema}{Remark}{Remarks}
\numberwithin{equation}{theo}

\newcommand{\N}{\mathbb{N}}								
\newcommand{\Z}{\mathbb{Z}}
\newcommand{\Q}{\ensuremath{\mathbb{Q}}}
\newcommand{\R}{\ensuremath{\mathbb{R}}}
\newcommand{\C}{\mathbb{C}}
\newcommand{\id}{\mathrm{id}}
\newcommand{\Can}[1]{\omega_{#1}}							
\newcommand{\codim}[2]{\mathrm{codim}_{#2}(#1)}				
\newcommand{\cc}[2]{\mathrm{c}_{#1}(#2)}					
\newcommand{\cct}[1]{\mathrm{c}(#1)}					
\newcommand{\ccorb}[2]{\mathrm{\tilde c}_{#1}(#2)}			
\newcommand{\ccmum}[2]{\mathrm{\hat c}_{#1}(#2)}			
\newcommand{\ccsm}[2]{\mathrm c^{\mathrm {SM}}_{#1}(#2)}			
\newcommand{\ccsmt}[1]{\mathrm c^{\mathrm {SM}}(#1)}			
\newcommand{\HH}[3]{\mathrm{H}^{#1}(#2,#3)}					
\newcommand{\HHc}[3]{\mathrm{H}_{\mathrm{c}}^{#1}(#2,#3)}	
\newcommand{\Hh}[3]{\mathrm{H}_{#1}(#2,#3)}					
\newcommand{\HBM}[2]{\mathrm{H}^{\mathrm{BM}}_{#1}(#2)}		
\renewcommand{\AA}[2]{\mathrm{A}^{#1}(#2)}
\newcommand{\opA}[2]{\mathrm{opA}^{#1}(#2)}
\newcommand{\Aa}[2]{\mathrm{A}_{#1}(#2)}
\newcommand{\Bb}[2]{\mathrm{B}_{#1}(#2, \R)}
\newcommand{\fund}[1]{[#1]}						
\newcommand{\NEbar}[1]{\overline{\mathrm{NE}}(#1)}			
\newcommand{\NN}[2]{\mathrm{N}^{#1}(#2)_\R}					
\newcommand{\Nn}[2]{\mathrm{N}_{#1}(#2)_\R}					
\newcommand{\TR}[1]{\mathrm{T}(#1)}						
\newcommand{\PH}[1]{\mathrm{PH}_{#1,\R}}					
\newcommand{\KK}[1]{\mathscr{K}^1_{#1,\R}}					
\newcommand{\Cinfty}[1]{\mathscr{C}^\infty_{#1,\R}}			
\newcommand{\Per}{\mathrm{P}}								
\newcommand{\Kcone}[1]{\mathcal{K}_{#1}}					
\newcommand{\piet}[1]{\pi_1^{\text{\'et}}(#1)}
\newcommand{\GL}[2]{\mathrm{GL}(#1, #2)}

\newcommand{\xysquare}[8]{
\xymatrix{
#1 \ar[r]^-{#5} \ar[d]_-{#6} & #2 \ar[d]^-{#7} \\
#3 \ar[r]_-{#8} & #4 }
}
\renewcommand{\O}[1]{\mathscr{O}_{#1}}						
\renewcommand{\mid}{:}
\renewcommand{\Re}{\mathrm{Re}}
\DeclareMathOperator\Pic{Pic}
\DeclareMathOperator\Gal{Gal}
\DeclarePairedDelimiter\paren{(}{)}
\DeclarePairedDelimiter\set{\{}{\}}
\DeclarePairedDelimiter\pair{\langle}{\rangle}
\DeclarePairedDelimiter\abs{\lvert}{\rvert}
\DeclarePairedDelimiterX\rest[2]{}{\big\vert_{#2}}{#1}

\newcommand{\eps}{\varepsilon}
\renewcommand{\epsilon}{\varepsilon}
\newcommand{\imp}{\Rightarrow}
\newcommand{\lto}{\longrightarrow}
\renewcommand{\P}{\mathbb P}
\newcommand{\x}{\times}
\renewcommand{\i}{\mathrm{i}}
\renewcommand{\phi}{\varphi}
\newcommand{\inj}{\hookrightarrow}
\newcommand{\bij}{\overset\sim\lto}
\newcommand{\isom}{\cong}
\newcommand{\tensor}{\otimes}
\DeclareMathOperator{\Hom}{\mathrm{Hom}}
\DeclareMathOperator{\Aut}{Aut}
\newcommand{\sg}{\mathrm{sg}}
\newcommand{\sm}{_\mathrm{sm}}
\newcommand{\kahler}{K\"ahler\xspace}
\DeclareMathOperator{\Alb}{Alb}
\newcommand{\alb}{\mathrm{alb}}
\newcommand{\wt}{\widetilde}
\newcommand{\wb}{\overline}
\newcommand{\wh}{\widehat}
\newcommand{\del}{\partial}
\newcommand{\delbar}{\overline\partial}
\newcommand{\from}{\colon}
\newcommand{\dual}{^{\smash{\scalebox{.7}[1.4]{\rotatebox{90}{\textup\guilsinglleft}}}}}
\newcommand{\acts}{\ \rotatebox[origin=c]{-90}{\ensuremath{\circlearrowleft}}\ }
\newcommand{\pd}{\mathrm{PD}}
\DeclareMathOperator{\Exc}{Exc}
\DeclareMathOperator{\supp}{supp}

\newcommand{\sF}{\mathscr{F}}

\newcommand{\sT}{\mathscr{T}}

\newdir{ ir}{{}*!/-5pt/@^{(}} 
\newdir{ il}{{}*!/-5pt/@_{(}} 

\newcommand{\factor}[2]{\left. \raise 2pt\hbox{$#1$} \right/\hskip -2pt\raise -2pt\hbox{$#2$}}

\begin{document}

\title{Finite quotients of three-dimensional complex tori}

\author{Patrick Graf}
\address{Lehrstuhl f\"ur Mathematik I, Universit\"at Bayreuth, 95440 Bayreuth, Germany}
\email{\href{mailto:patrick.graf@uni-bayreuth.de}{patrick.graf@uni-bayreuth.de}}
\urladdr{\href{http://www.pgraf.uni-bayreuth.de/en/}{http://www.pgraf.uni-bayreuth.de/en/}}

\author{Tim Kirschner}
\address{Fakult\"at für Mathematik, Universit\"at Duisburg--Essen, 45117 Essen, Germany}
\email{\href{mailto:tim.kirschner@uni-due.de}{tim.kirschner@uni-due.de}}
\urladdr{\href{http://timkirschner.tumblr.com}{http://timkirschner.tumblr.com}}

\date{May 6, 2017}




\begin{abstract}
We provide a characterization of quotients of three-dimen\-sional complex tori by finite groups that act freely in codimension one via a vanishing condition on the first and second orbifold Chern class.
We also treat the case of actions free in codimension two, using instead the ``birational'' second Chern class, as we call it.

Both notions of Chern classes are introduced here in the setting of compact complex spaces with klt singularities.
In such generality, this topic has not been treated in the literature up to now.
We also discuss the relation of our definitions to the classical Schwartz--MacPherson Chern classes.
\end{abstract}

\maketitle


\section{Introduction}

Consider an $n$-dimensional compact \kahler manifold $(X, \omega)$ with $\cc1X = 0$ and $\int_X \cc2X \wedge \omega^{n-2} = 0$.
The first condition implies, via Yau's solution to the Calabi conjecture, that $X$ can be equipped with a Ricci-flat \kahler metric~\cite{Yau78}.
As a consequence of the second condition, $X$ is then uniformized by $\C^n$, i.e.~the universal cover of $X$ is affine space.
Equivalently, $X$ is the quotient of a complex torus~$T$ by a finite group $G$ acting freely on $T$.

The philosophy of the Minimal Model Program (MMP) is that the most natural bimeromorphic models of a given \kahler manifold will in general have mild singularities~\cite{HP15}.
From this point of view, it is certainly important to extend the above result to singular complex spaces.
That is, one would like to have a criterion for a singular space $X$ to be the quotient of a complex torus by a finite group acting freely \emph{in codimension~one}.

This problem has attracted considerable interest in the past, but results are available only in the projective case, i.e.~for quotients of abelian varieties: see the article \cite{SBW94} by Shepherd-Barron and Wilson for the three-dimensional case, and the more recent ones~\cite{GKP13} by Greb, Kebekus, and Peternell and~\cite{LT14} by Lu and Taji in higher dimensions.
In this paper, we make a step towards settling the problem in general by proving the following uniformization results for \kahler threefolds with canonical singularities.

\begin{theo}[Characterization of three-dimensional torus quotients, I] \label{main thm I}
Let $X$ be a compact complex threefold with canonical singularities. The following are equivalent:
\begin{enumerate}
\item \label{mainI.1} We have $\cc1X = 0 \in \HH2X\R$, and there exists a \kahler class $\omega \in \HH2X\R$ such that $\ccorb2X \cdot \omega = 0$.
\item \label{mainI.2} There exists a $3$-dimensional complex torus $T$ and a holomorphic action of a finite group $G \acts T$, free in codimension one, such that $X \isom \factor TG$.
\end{enumerate}
\end{theo}

\noindent
Here $\ccorb2X$ denotes the \emph{second orbifold Chern class} of $X$, see \cref{c2orb}.

\begin{coro}[Characterization of three-dimensional torus quotients, II] \label{main thm II}
Let $X$ be a compact complex threefold with canonical singularities. The following are equivalent:
\begin{enumerate}
\item \label{mainII.1} We have $\cc1X = 0 \in \HH2X\R$, and there exists a \kahler class $\omega$ on $X$ as well as a resolution of singularities $f \from Y \to X$, minimal in codimension two, such that
\[ \int_Y \cc2Y \wedge f^*(\omega) = 0. \]
\item \label{mainII.2} There exists a $3$-dimensional complex torus $T$ and a holomorphic action of a finite group $G \acts T$, free in codimension \emph{two}, such that $X \isom \factor TG$.
\end{enumerate}
\end{coro}

\begin{rema*}
The second Chern class condition in \ref{mainII.1} is a way of saying ``$\cc2X \cdot \omega = 0$'' that does not involve showing independence of the choice of resolution $f \from Y \to X$.
In \cref{sec.Chern sg}, we discuss in detail both notions of second Chern class appearing above.
We also relate them to the classical Schwartz--MacPherson Chern classes (\cref{SM Chern classes}).
\end{rema*}

\begin{rema*}
Assume that $X$ is projective.
In \cite{SBW94} and the other references cited above, $\ccorb2X$ needs to intersect an ample Cartier divisor trivially, while for us it is sufficient to have a \kahler form with this property.
In this sense, our result is new even in the projective case.
Of course, a posteriori both conditions are equivalent, but this is precisely what we need to prove.
\end{rema*}

\subsection*{Further problems}

\cref{main thm I} does not yield a full characterization of torus quotients because quotient singularities are in general not canonical, but only klt.
Therefore it would be most natural to drop the a priori assumption on canonicity.
Also the restriction to dimension three should obviously not be necessary.
That said, we propose the following conjecture.

\begin{conj}[Characterization of torus quotients] \label{main conj}
Let $X$ be a compact complex space of dimension $n \ge 2$.
The following are equivalent:
\begin{enumerate}
\item \label{main conj.1} $X$ has klt singularities, $\cc1X = 0 \in \HH2X\R$, and there exists a \kahler class $\omega \in \HH2X\R$ such that $\ccorb2X \cdot \omega^{n-2} = 0$.
\item \label{main conj.2} There exists a complex torus $T$ and a holomorphic action of a finite group $G \acts T$, free in codimension one, such that $X \isom \factor TG$.
\end{enumerate}
\end{conj}

In dimension $n = 2$, this is well-known, see \cref{surface uniformization}.
In dimension $3$, the conjecture would follow from \cref{main thm I} and the following special case of the Abundance Conjecture (see \cref{main conj 3}):
\emph{Let $X$ be a compact \kahler threefold with klt singularities and $\cc1X = 0$.
Then the canonical sheaf of $X$ is torsion, that is, the $m$-th reflexive tensor power $\Can X^{[m]} \isom \O X$ for some $m > 0$.}
This is already known in important special cases, namely
\begin{itemize}
\item if $X$ has canonical singularities (due to Campana--H\"oring--Peternell, see~\cite[Proposition~10.2]{AlgApprox}), and
\item if $X$ is projective (in any dimension, due to Nakayama~\cite[Corollary~4.9]{Nak04}).
\end{itemize}
In dimensions $n \ge 4$ our methods do not seem to apply.
Cf.~\cref{higher-dim}, and note that also the Serre duality argument on \hyperlink{Serre}{p.~\pageref{Serre page}} breaks down in higher dimensions.

\subsection*{Outline of proof of~\cref{main thm I}}

The non-trivial direction of our result is of course ``\ref{mainI.1} $\imp$ \ref{mainI.2}''.
The first step is to observe that by abundance, the canonical sheaf of $X$ is torsion.
Taking an index one cover, we may assume that $X$ has trivial canonical sheaf.
We then distinguish two cases, according to whether $X$ is projective or not.
If $X$ is projective, we decompose $\HH2X\R$ into an algebraic and a transcendental part (\cref{decomp}) in order to replace the \kahler class $\omega$ intersecting $\ccorb2X$ trivially by the first Chern class of an ample $\R$-Cartier divisor.
Using Miyaoka's famous semipositivity theorem, that divisor can even be chosen to be Cartier, i.e.~with integral coefficients.
By the result of Shepherd-Barron and Wilson mentioned above, $X$ is then a quotient of an abelian threefold.

If $X$ is not projective, its Albanese map is a fibre bundle over a positive-dimensional complex torus.
The fibre $F$ has trivial canonical sheaf and at worst canonical singularities.
We calculate that its second orbifold Chern class vanishes, $\ccorb2F = 0$.
Since Yau's result extends to spaces with quotient singularities, $F$ can be equipped with a Ricci-flat \kahler metric and we obtain that the tangent bundle of the smooth locus of $F$ is flat.
Combined with a new result about \'etale fundamental groups of complex klt surfaces (\cref{piet}), this implies that $F$ is a torus quotient and then so is $X$.

\subsection*{Acknowledgements}

We would like to thank Daniel Greb for proposing this topic to us and for several fruitful discussions.
Part of this work was done while the first author visited the University of Duisburg--Essen.

\section{Basic conventions and definitions} \label{sec.notation}

All complex spaces are assumed to be separated, connected and reduced, unless otherwise stated.

\begin{defi}[Resolutions]
A \emph{resolution of singularities} of a complex space $X$ is a proper bimeromorphic morphism $f \from Y \to X$, where $Y$ is smooth.
\begin{enumerate}
\item We say that the resolution is \emph{projective} if $f$ is a projective morphism.
That is, $f$ factors as $Y \inj X \x \P^n \to X$, where the first map is a closed embedding and the second one is the projection.
In this case, if $X$ is compact \kahler then so is $Y$.
Any compact complex space $X$ has a projective resolution by~\cite[Thm.~3.45]{Kol07}.
\item A resolution is said to be \emph{strong} if it is an isomorphism over the smooth locus of $X$.
\item The resolution $f$ is said to be \emph{minimal} if it is projective and the canonical sheaf $\Can Y$ is $f$-nef.
This means that $\deg (\Can Y|_C) \ge 0$ for every compact curve $C \subset Y$ mapped to a point by $f$.
\item The resolution $f$ is said to be \emph{minimal in codimension two} if there exists an analytic subset $S \subset X$ with $\codim SX \ge 3$ such that for $U \coloneqq X \setminus S$, the restriction $f^{-1}(U) \to U$ is minimal.
\end{enumerate}
\end{defi}

For the definition of canonical and klt singularities we refer to~\cite[Def.~2.34]{KM98}.

\begin{nota}
\emph{Sheaf cohomology} is denoted by $\HH kX\sF$ as usual.
By $\HHc kX\sF$ we mean \emph{cohomology with compact support}, that is, the right derived functors of taking global sections with compact support.
We will also use \emph{homology} $\Hh kX\R$ and \emph{Borel--Moore homology} with integer coefficients $\HBM kX$.
The \emph{dual} of a (not necessarily finite-dimensional) real vector space $V$ is denoted $V \dual \coloneqq \Hom_\R(V, \R)$.
\end{nota}

\begin{defi}[Quasi-\'etalit\'e]
A finite surjective map $f \from Y \to X$ of normal complex spaces is said to be \emph{\'etale in codimension $k$} if for some open subset $X^\circ \subset X$ with $\codim {X \setminus X^\circ} X \ge k + 1$, the restriction $f^{-1}(X^\circ) \to X^\circ$ is \'etale.
The map $f$ is called \emph{quasi-\'etale} if it is \'etale in codimension one.
\end{defi}

\section{\kahler metrics on singular spaces}

In this section we collect several technical results about \kahler metrics and their cohomology classes on singular complex spaces.
The statements in this section are probably well-known to experts. Unfortunately, we have been unable to find published proofs of these results, at least not in the exact form we need.
Since our arguments in the rest of the paper depend crucially on these facts, and also for the reader's convenience, we have chosen to include full proofs here.
As far as notation is concerned, we mostly follow~\cite{Var89} and~\cite{Bin83}.

\subsection{Singular \kahler spaces}

First, we set up some notation and we define what a \kahler metric on a complex space is.

\begin{nota}[Pluriharmonic functions, \protect{\cite[pp.~17,~23]{Var89}}] \label{PH X}
Let $X$ be a reduced complex space.
We denote by $\Cinfty X$ the sheaf of smooth real-valued functions on $X$.
Moreover, we denote by $\PH X$ the image of the real part map $\Re \colon \O X \to \Cinfty X$, which is called the sheaf of real-valued \emph{pluriharmonic functions} on $X$, and we set $\KK X \coloneqq \factor{\Cinfty X}{\PH X}$.
\end{nota}

The sheaf $\PH X$ appears in two different short exact sequences:
\begin{equation} \label{K1 sequence}
0 \lto \PH X \lto \Cinfty X \lto \KK X \lto 0
\end{equation}
and
\begin{equation} \label{Re sequence}
0 \lto \underline\R_X \overset{\i\cdot}\lto \O X \xrightarrow{\,\Re\,} \PH X \lto 0.
\end{equation}
We denote by
\[ \delta^0 \from \KK X(X) \lto \HH1X{\PH X} \]
the connecting homomorphism in degree $0$ associated to the sequence (\ref{K1 sequence}),
and by
\[ \delta^1 \from \HH1X{\PH X} \lto \HH2X\R \]
the connecting homomorphism in degree $1$ associated to the sequence (\ref{Re sequence}).

\begin{rema}
\begin{enumerate}
\item\label{492} Since partitions of unity exist for the sheaf $\Cinfty X$, it is acyclic.
Hence the map $\delta^0$ is always surjective.
\item\label{474} If $X$ is compact and normal and the natural map $\HH1X\C \lto \HH1X{\O X}$ is surjective (e.g.~if $X$ is normal projective with Du Bois singularities), then one can show that also $\HH1X{\i\R} \to \HH1X{\O X}$ is surjective by a mixed Hodge structure argument.
In this case, $\delta^1$ is injective.
\item If $X$ is a compact \kahler manifold, \ref{474} can be made more precise: the map $\delta^1$ induces an isomorphism $\HH1X{\PH X} \bij \mathrm H^{1,1}(X) \cap \HH2X\R$.
\end{enumerate}
\end{rema}

\begin{nota}[Period class, \protect{\cite[p.~525]{Bin83}}] \label{period class}
We write
\[ \Per = \delta^1 \circ \delta^0 \colon \KK X(X) \lto \HH2X\R. \]
For an element $\kappa \in \KK X(X)$ we call $\Per(\kappa)$ the \emph{period class} of $\kappa$ on $X$.
\end{nota}

\begin{defi}[\kahler metrics, \protect{\cite[pp.~23,~18]{Var89}}] \label{kahler def}
Let $X$ be a reduced complex space.
\begin{enumerate}
\item A \emph{\kahler metric} on $X$ is an element $\kappa$ of $\KK X(X)$ which can be represented by a family $(U_i,\phi_i)_{i \in I}$ such that $\phi_i$ is a smooth strictly plurisubharmonic function on $U_i$ for all $i \in I$.
That is, locally $\phi_i$ is induced by a smooth strictly plurisubharmonic function on an open subset of $\C^{N_i}$ under a local embedding $U_i \inj \C^{N_i}$.
\item We say that $c \in \HH2X\R$ is a \emph{\kahler class} on $X$ if there exists a \kahler metric $\kappa$ on $X$ such that $c = \Per(\kappa)$.
\item We say that $X$ is \emph{\kahler} if there exists a \kahler metric on $X$.
\end{enumerate}
\end{defi}

\subsection{Properties of \kahler metrics}

Our first proposition is the \kahler analog of a well-known property of ample line bundles.

\begin{prop}[Finite pullbacks] \label{kahler finite pullback}
Let $f\colon Y \to X$ be a finite morphism of complex spaces, $c \in \HH2X\R$ a \kahler class on $X$. Then $f^*(c)$ is a \kahler class on $Y$.
\end{prop}

\begin{proof}
Since $c$ is \kahler class on $X$, there exists a \kahler metric $\kappa$ on $X$ such that $\Per(\kappa) = c$.
By \cite[p.~253, Claim A in the proof of Theorem 1]{Vaj96}, there exists a smooth function $\phi \colon Y \to \R$ such that $f^*(\kappa) + \phi$ is a \kahler metric on $Y$.
Here, $f^* \colon \KK X \to f_*(\KK Y)$ is the sheaf map induced by the pullback of smooth functions $\Cinfty X \to f_*(\Cinfty Y)$.
Explicitly this means there exists a family $(U_i,\kappa_i)_{i \in I}$ of smooth strictly plurisubharmonic functions on $X$ representing $\kappa$ as well as an open cover $(V_j)_{j \in J}$ of $Y$ and a map $\lambda \colon J \to I$ such that for all $j \in J$ we have $V_j \subset f^{-1}(U_{\lambda(j)})$ and $\rest{\kappa_{\lambda(j)} \circ f}{V_j} + \rest\phi{V_j}$ is a strictly plurisubharmonic function on $V_j$.

Since $\phi \in \Cinfty Y(Y)$, we see that $\delta^0(\phi) = 0$. Hence, given that the period class map $\Per$ commutes with pullback along $f$,
\[ \Per \big( f^*(\kappa) + \phi \big) = \Per\big( f^*(\kappa) \big) = f^* \big( \Per(\kappa) \big) = f^*(c), \]
so $f^*(c)$ is a \kahler class.
\end{proof}

The next two results concern openness properties of \kahler metrics.

\begin{prop}[Being \kahler is an open property, I] \label{kahler open 1}
Let $X$ be a compact complex space, $\kappa$ a \kahler metric on $X$, and $\phi \in \KK X(X)$ an arbitrary element.
Then there exists a number $\epsilon > 0$ such that for all $t \in \R$ with $\abs t < \epsilon$, we have that $\kappa + t\phi$ is a \kahler metric on $X$.
\end{prop}

\begin{proof}
Since $\kappa$ is a \kahler metric, there exists a family $(U_i, \kappa_i)_{i \in I}$ representing $\kappa$ in the quotient sheaf $\KK X$ such that $(U_i)_{i \in I}$ is an open cover of $X$ and $\kappa_i$ is strictly plurisubharmonic on $U_i$ for all $i \in I$.
Likewise, $\phi$ is represented by a family $(V_j, \phi_j)_{j \in J}$ where $\phi_j \in \Cinfty X(V_j)$ for all $j \in J$.
Passing to a common refinement and using the definitions of $\Cinfty X$ and strict plurisubharmonicity on $X$, we may assume that
\begin{itemize}
\item $(U_i)_{i \in I} = (V_j)_{j \in J}$, 
\item there exist $W_i \subset \C^{N_i}$ open, $N_i \in \N$, and closed embeddings $g_i \colon U_i \inj W_i$,
\item there exist $\tilde\kappa_i, \tilde\phi_i \in \mathscr C^\infty(W_i,\R)$ such that the $\tilde\kappa_i$ are strictly plurisubharmonic functions, $[\tilde\kappa_i] = \kappa_i$, and $[\tilde\phi_i] = \phi_i$,
\item there exist relatively compact, open subsets $W_i' \Subset W_i$ such that, setting $U_i' \coloneqq g_i^{-1}(W_i')$, we have $\bigcup_{i \in I} U_i' = X$.
\end{itemize}
Furthermore, since $X$ is compact, we may assume that $I$ is finite.

Let $i \in I$. Then, since strict plurisubharmonicity for smooth functions on $\C^{N_i}$ is equivalent to their Levi form being positive definite at each point, there exists $\eps_i > 0$ such that $\rest{(\tilde\kappa_i + t\tilde\phi_i)}{W_i'}$ is strictly plurisubharmonic on $W_i'$ for all $t \in \R$ with $\abs t < \eps_i$.
Define $\eps \coloneqq \min\set{\eps_i \mid i \in I} > 0$.
When $t$ is a real number such that $\abs t < \eps$, then $\kappa + t \phi$ is a \kahler metric on $X$, for it is represented by the family $\big( U_i', \rest{(\kappa_i + t \phi_i)}{U_i'} \big)$ and $\rest{(\kappa_i + t \phi_i)}{U_i'}$ is induced by $\rest{(\tilde\kappa_i + t\tilde\phi_i)}{W_i'}$ for all $i \in I$.
\end{proof}

\begin{coro}[Being \kahler is an open property, II] \label{kahler open 2}
Let $X$ be a compact complex space. Then
\begin{enumerate}
\item \label{h11 finite} the real vector space $\HH1X{\PH X}$ is finite-dimensional, and
\item \label{kahler cone open} the set $\Kcone X \coloneqq \set{\delta^0(\kappa) \mid \kappa \text{ is a \kahler metric on } X}$ is an open convex cone in $\HH1X{\PH X}$, called the \emph{\kahler cone} of $X$.
\end{enumerate}
\end{coro}

\begin{proof}
\ref{h11 finite} is clear by looking at the long exact sequence in cohomology associated to the short exact sequence of sheaves \eqref{Re sequence}.

If $\kappa$ and $\lambda$ are \kahler metrics on $X$, then $s\kappa + t\lambda$ is a \kahler metric on $X$ for all $s,t \in \R_{\ge0}$ with $s+t > 0$, since the analogous statement holds for strictly plurisubharmonic functions on $X$.
Therefore the set of \kahler metrics on $X$ is a convex cone in the real vector space $\KK X(X)$.
Since $\delta^0$ is a linear map, $\Kcone X$ is a convex cone in $\HH1X{\PH X}$.

For openness, consider an arbitrary element $c = \delta^0(\kappa) \in \Kcone X$, for $\kappa$ a \kahler metric on $X$.
By \ref{h11 finite}, there exists a finite basis $(b_1, \dots, b_\rho)$ for $\HH1X{\PH X}$.
As we remarked in \ref{492}, the map $\delta^0 \colon \KK X(X) \to \HH1X{\PH X}$ is surjective.
Hence, there exist $\phi_1, \dots, \phi_\rho \in \KK X(X)$ such that $\delta^0(\phi_i) = b_i$ for all $1 \le i \le \rho$.
By \cref{kahler open 1}, there exists a number $\epsilon > 0$ such that for all $1 \le i \le \rho$ and all $t \in \R$ with $\abs t < \epsilon$ we have that $\kappa + t \phi_i$ is a \kahler metric on $X$.
Consequently $c + t b_i \in \Kcone X$ for all $i$ and $t$ as before.
Since $\Kcone X$ is a convex cone, we deduce that $c + \sum_{i=1}^\rho t_ib_i \in \Kcone X$ for all $t = (t_i) \in \R^\rho$ with $\abs{t_i} < \epsilon/\rho$ for all $i$, and we obtain~\ref{kahler cone open}.
\end{proof}

\section{A decomposition of the second cohomology group} \label{sec.decomp}

The purpose of this section is to associate to any \kahler class on a mildly singular projective variety $X$ an $\R$-ample divisor class having the same intersection numbers with all curves in $X$ (\cref{kahler kleiman}).
Before we can give the statement, we need to introduce some notation.

\begin{nota}
Let $X$ be a non-empty complex space, of pure dimension $n$.
\begin{enumerate}
\item We denote by $\fund X \in \HBM{2n}X$ the \emph{fundamental class} of $X$ in Borel--Moore homology.
\item \label{item class} If $i \from A \inj X$ is a nonempty purely $k$-dimensional closed analytic subset, abusing notation we write $\fund A$ too for $i_* \fund A \in \HBM{2k}X$.
\end{enumerate}
Assume furthermore that $X$ is compact.
In this case, $\Hh*X\Z = \HBM*X$.
\begin{enumerate}[resume]
\item\label{Bb2kX} For any integer $k \ge 0$, we define $\Bb{2k}X \subset \Hh{2k}X\R$ to be the real linear subspace spanned by the set of all $\fund A \in \Hh{2k}X\R$, where $A \subset X$ is a $k$-dimensional irreducible closed analytic subset.
\item $\NN1X \subset \HH2X\R$ is the real linear subspace spanned by the image of the first Chern class map $\HH1X{\O X^*} \to \HH2X\R$.
\item $\TR X \subset \HH2X\R$ is the subspace orthogonal to $\Bb2X$ with respect to the canonical pairing $\pair{ \_, \_ } \from \HH2X\R \times \Hh2X\R \to \R$. That is,
\[ \TR X = \big\{ a \in \HH2X\R \;\big|\; \forall b \in \Bb2X \colon \pair{ a, b } = 0 \big\}. \]
\item $\Nn1X$ is the quotient of $\Bb2X$ by \emph{numerical equivalence}.
That is, $\Nn1X = \factor{\Bb2X}{\Bb2X \cap S}$ where $S \subset \Hh2X\R$ is the subspace orthogonal to $\NN1X$ with respect to $\pair{ \_, \_ }$.
\end{enumerate}
\end{nota}

\begin{prop}[Decomposition of singular cohomology] \label{decomp}
Let $X$ be a projective variety with \emph{$1$-rational} singularities only. Then
\[ \HH2X\R = \NN1X \oplus \TR X. \]
\end{prop}

For the definition of $1$-rational singularities, see~\cite[Proposition~1.2]{Nam02}.
For our purposes, it is sufficient to know that rational singularities are $1$-rational.
In particular, this applies to canonical (or more generally klt) singularities.

\begin{proof}
First, we show that $\NN1X \cap \TR X = \set0$.
To this end, let $a \in \NN1X$ be an element such that $\pair{a, b} = 0$ for all $b \in \Bb2X$.
By definition, we may write $a = \sum_{i=1}^k a_i \cc1{L_i}$ with $a_1, \dots, a_k \in \R$, $L_1, \dots, L_k \in \Pic(X)$, and $k \in \N$.
By a linear algebra argument\footnote{
Write the condition ``$a \in \TR X$'' as a finite system of linear equations over $\Q$ in the $a_i$ and note that any real solution to such a system is a real linear combination of rational solutions.
See also the proof of \cite[Proposition~1.3.13]{Laz04a}.}, there are line bundles $M_1, \dots, M_\ell \in \Pic(X)$ such that $\cc1{M_j} \in \TR X$ for all $1 \le j \le \ell$ and $a = \sum_{j=1}^\ell b_j \cc1{M_j}$ for some $b_1, \dots, b_\ell \in \R$.
Now \cite[Corollary~1.4.38]{Laz04a} implies that there exist integers $N_j > 0$ such that $M_j^{\tensor N_j} \in \Pic^0(X)$, i.e.~$M_j^{\tensor N_j}$ is a deformation of $\O X$. In particular, $\cc1{M_j} = 0$ for each $j$ and then clearly $a = 0$.

To conclude, it suffices to show that $\dim \NN1X + \dim \TR X \ge \dim H^2(X, \R)$.
By~\cite[Proposition~1.2]{Nam02}, an element $b \in \Bb2X$ is zero if $\pair{ a, b } = 0$ for all $a \in \NN1X$.
In other words, $\Nn1X = \Bb2X$ and hence the map $\Bb2X \to \NN1X \dual$ induced by $\pair{ \_, \_ }$ is injective.
Thus
\begin{align*}
\dim \NN1X + \dim \TR X & \ge \dim \Bb2X + \dim \TR X \\
& = \dim H^2(X, \R),
\end{align*}
the last equality being due to the orthogonality of $\Bb2X$ and $\TR X$ with respect to the perfect pairing $\pair{ \_, \_ }$.
\end{proof}

\begin{lemm}[Pullback of transcendental classes] \label{pullback T}
Let $f \colon Y \to X$ be a morphism between compact complex spaces.
Then $f^*(\TR X) \subset \TR Y$.
\end{lemm}

\begin{proof}
Let $a \in \TR X$ be arbitrary and $D \subset Y$ an irreducible reduced closed complex subspace of dimension $1$.
Then, by Remmert's mapping theorem, $f(D) \subset X$ too is an irreducible reduced closed complex subspace.
Moreover, either $\dim f(D) = 0$ or $\dim f(D) = 1$.
If $f(D)$ is $0$-dimensional, then clearly $\pair{ f^* a, \fund D } = 0$.
If $f(D)$ is $1$-dimensional, then there exists a number $d > 0$ such that $D \to f(D)$ is a $d$-sheeted analytic covering.
Therefore $\pair{ f^* a, \fund D } = d \cdot \pair{ a, \fund{f(D)} } = 0$.
We conclude that $f^* a \in \TR Y$.
\end{proof}

\begin{prop}[Algebraic part of a \kahler class] \label{kahler kleiman}
Let $X$ be a projective variety with $1$-rational singularities, $c \in \HH2X\R$ a \kahler class on $X$.
Write $c = h + t$ according to the direct sum decomposition of \cref{decomp}.
Then $h \in \NN1X$ is an \R-ample divisor class.
\end{prop}

\begin{proof}
Let $\NEbar X \subset \Nn1X$ be the cone of curves.
By Kleiman's ampleness criterion \cite[Theorem 1.4.29]{Laz04a}, it suffices to show that $h \cdot a > 0$ for all $a \in \NEbar X \setminus \set0$.
But $h \cdot a = c \cdot a - t \cdot a = c \cdot a$, so we only need to show that $c \cdot a > 0$.

To begin with, we remark that $c \cdot a \ge 0$, since clearly $c \cdot \fund C > 0$ for any irreducible and reduced curve $C \subset X$.
Now since $a \ne 0$ in $\Nn1X$, there exists a line bundle $L \in \Pic(X)$ such that $\cc1L \cdot a \ne 0$.
We may assume that $\cc1L \cdot a > 0$.
We know \cite[(4.15)]{Bin83} that there exists a group homomorphism $\ell \colon \HH1X{\O X^*} \to \HH1X{\PH X}$ such that the following diagram commutes:
\[ \xysquare {\HH1X{\O X^*}} {\HH2X\Z} {\HH1X{\PH X}} {\HH2X\R.} {\mathrm{c}_1} \ell {} {\delta^1} \]
Since $c$ is a \kahler class on $X$, there exists a \kahler metric $\kappa$ on $X$ with the property $\Per(\kappa) = c$.
In particular, $\delta^0(\kappa) \in \Kcone X$.
By openness of the \kahler cone~\ref{kahler cone open}, there is a number $t<0$ such that $\delta^0(\kappa) + t \ell(L) \in \Kcone X$.
As a consequence,
\[ 0 \le \delta^1 \big( \delta^0(\kappa) + t \ell(L) \big) \cdot a = \paren*{\Per(\kappa) + t \delta^1(\ell(L))} \cdot a = \paren*{c + t \cc1L} \cdot a. \]
Thus
\[ 0 < -t \cc1L \cdot a \le c \cdot a, \]
which was to be demonstrated.
\end{proof}

\section{Chern classes on singular spaces} \label{sec.Chern sg}

In order to prove our main result, we need to discuss the Chern classes of the tangent sheaf of a singular \kahler space $X$ and their intersection numbers with a given \kahler class.
On complex manifolds there is a unique notion of Chern classes, but in the singular case there are at least two competing approaches:
Firstly, if $X$ has quotient singularities in sufficiently high codimension, one can define the ``orbifold'' Chern classes of $X$.
Secondly, one may pull back everything to an appropriate resolution of singularities $\wt X \to X$ to define a ``birational'' notion of Chern classes.
For our purposes, both approaches will be useful.

References for these matters include~\cite{SBW94, GKP13, LT14, CHP15}, but they all either make assumptions on smoothness in high codimension that are not satisfied in our setting, or they define intersection numbers with ample divisors only and not with arbitrary \kahler classes.
Therefore we have chosen to include here a self-contained presentation of the material.
We restrict ourselves to considering the first and second Chern class of a space with klt singularities, which is more than sufficient for this paper.
These notions should be of independent interest.

\begin{defi}[First Chern class] \label{c1}
Let $X$ be a normal complex space which is \Q-Gorenstein, i.e.~for some $m > 0$ the reflexive tensor power $\omega_X^{[m]} \coloneqq \big( \omega_X^{\tensor m} \big)^{**}$ is invertible, where $\omega_X$ is the dualizing sheaf.
The \emph{first Chern class} of $X$ is the cohomology class
\[ \cc1X \coloneqq - \frac 1 m \mathrm c_1 \big( \omega_X^{[m]} \big) \in \HH2X\R. \]
This is independent of the choice of $m$.
\end{defi}

\noindent
Spaces with klt singularities are \Q-Gorenstein by definition (see \cref{sec.notation}), so \cref{c1} applies to them.

Both our notions of second Chern class will be elements of $\HH{2n-4}X\R \dual = \Hom_\R \big( \HH{2n-4}X\R, \R \big)$, that is, linear forms on the appropriate cohomology group.
Of course, this is the same as giving a homology class, but this is not how we usually think about Chern classes.

\begin{defi}[``Orbifold'' second Chern class] \label{c2orb}
Let $X$ be a compact complex space with klt singularities, of pure dimension $n$.
Let $X^\circ \subset X$ be the (open) locus of quotient singularities of $X$.
The \emph{second orbifold Chern class} of $X$ is the unique element $\ccorb2X \in \HH{2n-4}X\R \dual$ whose restriction to $\HHc{2n-4}{X^\circ}\R \dual$ is the Poincar\'e dual of the second orbifold Chern class $\ccorb2{X^\circ} \in \HH4{X^\circ}\R$.
\end{defi}

In \cref{c2orb explain} below, we will discuss more carefully why this definition makes sense.
Using de Rham cohomology, we will also interpret it in terms of integrating differential forms.

\begin{propdefi}[``Birational'' second Chern class] \label{c2bir}
Let $X$ be a compact complex space with klt singularities, of pure dimension $n$.
Then there exists a resolution of singularities $f \from Y \to X$ which is minimal in codimension two.
The \emph{birational second Chern class} of $X$ is the element $\cc2X \in \HH{2n-4}X\R \dual$ defined by
\[ \cc2X \cdot a \coloneqq \int_Y \cc2Y \cup f^*(a) \quad \text{for any $a \in \HH{2n-4}X\R$,} \]
where $\cc2Y \in \HH4Y\R$ is the usual second Chern class of the complex manifold $Y$.
This definition is independent of the resolution $f$ chosen (provided it is minimal in codimension two).
\end{propdefi}

In particular, for classes $a_1, \dots, a_{n-2} \in \HH2X\R$ we may set
\[ \ccorb2X \cdot a_1 \cdots a_{n-2} \coloneqq \ccorb2X \cdot (a_1 \cup \cdots \cup a_{n-2}) \]
and likewise for $\cc2X$.
In this way, $\ccorb2X$ and $\cc2X$ yield $(n-2)$-multilinear forms on $\HH2X\R$.

\begin{rema}[Comparison of $\mathrm{\tilde c}_2$ and $\mathrm c_2$] \label{c2 comparison}
Let $X_1$ be a complex $2$-torus, and consider the quotient map $g \from X_1 \to X = \factor{X_1}{\pm 1}$.
Then $\ccorb2X = 0$, as follows from (\ref{c2orb cover}) below.
But $\cc2X = 24$ under the natural identification $\HH0X\R \dual = \R$, since the minimal resolution of $X$ is a K3 surface.
This shows that the two notions of Chern classes do not agree even on spaces with only canonical quotient singularities.

On the other hand, $\ccorb2X = \cc2X$ whenever $X$ is smooth in codimension two.
This can be seen as follows.
Let $i \from U \inj X$ be the smooth locus and let $f \from Y \to X$ be a strong resolution, i.e.~$V \coloneqq f^{-1}(U) \xrightarrow{f_U} U$ is an isomorphism.
Any $a \in \HH{2n-4}X\R$ can be written uniquely as $a = i_*(b)$ for some $b \in \HHc{2n-4}U\R$ (see the long exact sequence in \cref{c2orb explain}).
Then, with $j \from V \inj Y$ the inclusion,
\begin{align*}
\ccorb2X \cdot a & = \textstyle\int_U \cc2U \cup b && \text{almost by definition} \\
 & = \textstyle\int_V \cc2V \cup f_U^*(b) && \text{since $f_U$ is an isomorphism} \\
 & = \textstyle\int_V j^* \cc2Y \cup j^*(f^* a) && \text{since $\cc2V = j^* \cc2Y$ and $b = i^*(a)$} \\
 & = \textstyle\int_Y \cc2Y \cup f^*(a) && \text{since $Y \setminus V \subset Y$ is analytic} \\
 & = \cc2X \cdot a && \text{by definition.}
\end{align*}
\end{rema}

\begin{rema}[Schwartz--MacPherson Chern classes] \label{SM Chern classes}
The earliest treatment of Chern classes on singular varieties is due to Schwartz and MacPherson~\cite{Schwartz65, MacPherson74}.
It is natural to ask how our Chern classes relate to theirs.
Note that in~\cite{MacPherson74}, the construction is carried out only for compact complex \emph{algebraic} varieties $X$, and that the $k$-th Schwartz--MacPherson Chern class $\ccsm kX$ lives in $\Hh{2n-2k}X\Z$, where $n = \dim X$.
This is however not a problem, since we may use the isomorphism $\Hh{2n-4}X\R \isom \HH{2n-4}X\R \dual$ from the universal coefficient theorem to compare $\ccsm2X$ (with real coefficients) to $\cc2X$ and $\ccorb2X$.

We give an example where both $\ccorb2X \ne \ccsm2X$ and $\cc2X \ne \ccsm2X$.
Let $X = \factor{X_1}{\pm 1}$ be as in \cref{c2 comparison}, but assume that $X_1$ is algebraic, i.e.~an abelian surface.
Let $f \from \wt X \to X$ be the minimal resolution, where $\wt X$ is a K3 surface.
Denote $p_1, \dots, p_{16} \in X$ the sixteen singular points of $X$, corresponding to the $2$-torsion points of $X_1$.
If $i_k \from \set{p_k} \inj X$ is the inclusion, then using notation from~\cite[Proposition~1]{MacPherson74} we have
\[ \mathbf 1_X = f_* \big( \mathbf 1_{\wt X} \big) - \sum_{k=1}^{16} i_{k*} \big( \mathbf 1_{\set{p_k}} \big). \]
By~\cite[proof of Proposition~2]{MacPherson74} it follows that the total Chern class of $X$ in homology is
\[ \ccsmt X = f_* \big( \pd \big( \cct{\wt X} \big) \big) - \sum_{k=1}^{16} i_{k*} \big( \pd \big( \cct{\set{p_k}} \big) \big) \in \Hh*X\Z, \]
where $\pd \from \HH i{\wt X}\Z \to \Hh{4-i}{\wt X}\Z$ is the Poincar\'e duality map.
We obtain the second Chern class by looking at the degree $0$ part:
\[ \ccsm2X = f_*(24) - \sum_{k=1}^{16} i_{k*}(1) = 24 - 16 = 8 \in \Hh0X\Z. \]
Hence also $\ccsm2X = 8$ as an element of $\HH0X\R \dual = \R$.
Comparing this to \cref{c2 comparison}, we see that even on algebraic varieties with only canonical quotient singularities, $\ccsm2X$, $\ccorb2X$ and $\cc2X$ are three pairwise distinct notions.

On the other hand, $\ccsm2X = \cc2X$ if $X$ is algebraic and smooth in codimension two, and so all three versions coincide in this case:
Let $f \from \wt X \to X$ be a strong resolution.
Then we can write
\[ \mathbf 1_X = f_* \big( \mathbf 1_{\wt X} \big) - \sum_k g_{k*} \big( \mathbf 1_{Y_k} \big), \]
where $g_k \from Y_k \to X$ are maps from smooth varieties $Y_k$ of dimension $\dim_\C Y_k \le n - 3$, $n = \dim_\C X$.
Arguing as before and taking the degree $2n - 4$ part, we get
\[ \ccsm2X = f_* \big( \pd \big( \cc2{\wt X} \big) \big) \]
since $\Hh{2n-4}{Y_k}\Z = 0$ for all $k$.
Rewriting this statement in cohomology yields the claim, since $f$ is in particular minimal in codimension two.
\end{rema}

\begin{prop}[Behavior of $\mathrm{\tilde c}_2$ and $\mathrm c_2$ under quasi-\'etale maps] \label{c2orb/bir behaviour}
Let $g \from X_1 \to X$ be a finite surjective map between normal compact complex spaces of pure dimension~$n$, where $X$ has klt singularities.
Assume that $g$ is \'etale in codimension one.
Then also $X_1$ has klt singularities and for all $a \in \HH{2n-4}X\R$ we have
\begin{equation} \label{c2orb cover}
\ccorb2{X_1} \cdot g^*(a) = \deg(g) \cdot \ccorb2X \cdot a.
\end{equation}
If $g$ is \'etale in codimension two, then we also have
\begin{equation} \label{c2bir cover}
\cc2{X_1} \cdot g^*(a) = \deg(g) \cdot \cc2X \cdot a  
\end{equation}
for all $a \in \HH{2n-4}X\R$.
\end{prop}

\begin{rema}
For the map $g \from X_1 \to X$ from \cref{c2 comparison}, the left-hand side of (\ref{c2bir cover}) is zero, while the right-hand side evaluates to $48$.
This shows that (\ref{c2bir cover}) fails if $g$ is only \'etale in codimension one.
\end{rema}

\subsection{Auxiliary results}

Before we prove the above propositions, we collect some preliminary lemmas.

\begin{lemm}[Local structure of klt singularities] \label{local structure}
Let $X$ be a complex space with klt singularities.
Then there exists an analytic subset $Z$ in $X$ such that $\codim ZX \ge 3$ and such that for all $x \in X \setminus Z$, either $X$ is smooth at $x$ or there exist a klt surface singularity $(S, o)$ and an integer $n \ge 0$ for which we have $(X, x) \cong (S, o) \times (\C^n, 0)$ as germs of complex spaces.
In particular, $X \setminus Z$ has quotient singularities.
Furthermore, if $X$ has canonical singularities then $X \setminus Z$ is Gorenstein.
\end{lemm}

\begin{proof}
Assuming that $X$ is quasi-projective, \cite[Proposition 9.3]{GKKP11} shows the existence of a closed analytic set $Z \subset X$ of codimension greater than or equal to $3$ such that, for all $x \in X \setminus Z$, the germ $(X, x)$ is a quotient singularity.
The proof of the cited result, however, shows more precisely that either $(X, x)$ is smooth or $(X, x) \cong (S, o) \times (\C^n, 0)$ for a klt surface singularity $(S, o)$ and an integer $n \ge 0$.
If $X$ has canonical singularities, $(S, o)$ will be a canonical singularity too.
By \cite[Theorem 4.20]{KM98}, every such $(S, o)$ is a hypersurface singularity, whence $(X, x)$ is Gorenstein.

In the general case, the arguments of \cite{GKKP11} go through with minor modifi\-cations---for example, in order to obtain \cite[Proposition 2.26]{GKKP11} (``projection to a subvariety'') we need to employ the Open Projection Lemma \cite[p.~71]{CAS} instead of Noether normalization.
\end{proof}

\begin{lemm}[Pullback lemma] \label{cc pullback lemma}
Let $X$, $Y$, $Y_1$ be complex spaces, $h \colon Y_1 \to Y$ and $f \colon Y \to X$ proper morphisms, and $U$ an open subspace of $X$ with analytic complement such that the restriction of $h$ to $V_1 \coloneqq h^{-1}(f^{-1}(U))$ yields a $d$-fold étale covering $h_V \colon V_1 \to V \coloneqq f^{-1}(U)$. Assume that $Y$ and $Y_1$ are pure $n$-dimensional complex manifolds. Then, for all natural numbers $k \le \codim{X \setminus U}X - 1$ and all $a \in \HHc{2n-2k}X\R$, we have
\[ \int_{Y_1} \cc k{Y_1} \cup h^*(f^*(a)) = d \cdot \int_Y \cc kY \cup f^*(a). \]
\end{lemm}

\begin{proof}
Let $i \from U \inj X$, $j \from V \inj Y$ and $j_1 \from V_1 \inj Y_1$ be the inclusions and consider the following commutative diagram, where we write $h_*$ for $(h^*) \dual$ and analogously for the other maps.
\[ \xymatrix{
\HH{2k}{Y_1}\R \ar^-\pd[r] \ar[d] & \HHc{2n-2k}{Y_1}\R \dual \ar^-{(j_{1*})\dual}[r] \ar^-{h_*}[d] & \HHc{2n-2k}{V_1}\R \dual \ar^-{h_{V*}}[d] \\
\HH{2k}Y\R \ar^-\pd[r] & \HHc{2n-2k}Y\R \dual \ar^-{(j_*)\dual}[r] \ar^-{f_*}[d] & \HHc{2n-2k}V\R \dual \ar^-{f_{U*}}[d] \\
& \HHc{2n-2k}X\R \dual \ar^-{(i_*)\dual}[r] & \HHc{2n-2k}U\R \dual.
} \]
Here $\pd$ denotes the Poincar\'e duality isomorphism on the complex manifolds $Y$ and $Y_1$, respectively (\cref{pd orb} below).
Since $h_V$ is \'etale, any hermitian metric on $V$ pulls back to an hermitian metric on $V_1$ and so $h^* \cc kV = \cc k{V_1} \in \HH{2k}{V_1}\R$.
Furthermore, $\int_{V_1} h_V^* \sigma = d \cdot \int_V \sigma$ for all $\sigma \in \HHc{2n}V\R$.
Thus we see that
\[ h_{V*} \big( \pd \big( \cc k{V_1} \big) \big) = d \cdot \pd \big( \cc kV \big). \]
By a similar argument, $(j_*)\dual \big( \pd \big( \cc kY \big) \big) = \pd \big( \cc kV \big)$ and analogously for $j_1$.
This shows that the classes $\cc k{Y_1} \in \HH{2k}{Y_1}\R$ and $d \cdot \cc kY \in \HH{2k}Y\R$ in the left-hand side column are mapped to the same element in $\HHc{2n-2k}V\R \dual$.
In particular, they are mapped to the same element in $\HHc{2n-2k}U\R \dual$.
But then their images in $\HHc{2n-2k}X\R \dual$ are also the same, because $(i_*) \dual$ is injective.
The latter claim follows from the long exact sequence in compactly supported cohomology associated to the inclusions $i \from U \inj X$ and $\iota \from X \setminus U \inj X$~\cite[III.7.6]{Iversen86},
\[ \cdots \lto \HHc{2n-2k}U\R \xrightarrow{\; i_* \;} \HHc{2n-2k}X\R \xrightarrow{\; \iota^* \;} \underbrace{\HHc{2n-2k}{X \setminus U}\R}_{=0} \lto \cdots, \]
the last vanishing being due to the fact that $\dim_\R(X \setminus U) \le 2n - 2k - 2$.
We have thus shown that $f_* h_* \cc k{Y_1} = d \cdot f_* \cc kY$.
This is exactly the claim of the lemma.
\end{proof}

\subsection{Explanation of \cref{c2orb}} \label{c2orb explain}

As far as complex spaces with quotient singularities (``V-manifolds'', ``orbifolds'') are concerned, we use the terminology of \cite{Satake56}.
In particular, by a \emph{local uniformization} (``orbifold chart'') of an open subset $U$ of such a space $X$ we mean a triple $(\wt U, G, \phi)$, where
\begin{itemize}
\item $\wt U \subset \C^n$ is a contractible open set,
\item $G$ is a finite group acting on $\wt U$ linearly and freely in codimension one,
\item $\phi \from \wt U \to U$ is a $G$-invariant map exhibiting $U$ as the quotient $\wt U / G$.
\end{itemize}
We will use Poincar\'e duality in the following guise.

\begin{prop}[Poincar\'e duality for orbifolds] \label{pd orb}
Let $U$ be an $n$-dimensional connected complex space with quotient singularities.
Then for any $0 \le k \le 2n$, the bilinear pairing
\[ \HH k{X^\circ}\R \times \HHc{2n-k}{X^\circ}\R \lto \HHc{2n}{X^\circ}\R \isom \R \]
gives rise to an isomorphism $\pd \from \HH k{X^\circ}\R \bij \HHc{2n-k}{X^\circ}\R \dual$.
In terms of de Rham cohomology, the pairing is given by
\[ \big( [\alpha], [\beta] \big) \longmapsto \int_U \alpha \wedge \beta. \]
\end{prop}

\begin{proof}
This is a special case of Verdier duality~\cite[V.2.1]{Iversen86}.
If $U$ is a topological manifold, it is explained in~\cite[proof of~V.3.2]{Iversen86} how to deduce our statement from Verdier duality.
A closer look at the proof reveals that the only property of $U$ being used there is that every point $x \in U$ has a neighborhood basis consisting of open sets $V$ satisfying
\[ \HHc kV\R \isom \begin{cases}
\R, & k = 2n, \\
0, & \text{otherwise.}
\end{cases} \]
This continues to hold if $U$ has at worst quotient singularities, since using local uniformizations $(\wt V, G, \phi)$ of $V$ we have $\HHc kV\R = \HHc k{\wt V}\R ^G$.

For the statement about de Rham cohomology, see~\cite[Theorem~3]{Satake56}.
\end{proof}

Now consider the setting of \cref{c2orb}.
We start with the second orbifold Chern class $\ccorb2{X^\circ} \in \HH4{X^\circ}\R$, defined differential-geometrically as explained for example in~\cite[Section~2.B.1]{LT14}.
Its Poincar\'e dual $\pd(\ccorb2{X^\circ})$ is an element of $\HHc{2n-4}{X^\circ}\R \dual$.
Consider the inclusions
\[ i \from X^\circ \inj X \quad\text{and}\quad \iota \from Z = X \setminus X^\circ \inj X \]
and the following excerpt from the associated long exact sequence:
\[ \HH{2n-5}Z\R \lto \HHc{2n-4}{X^\circ}\R \xrightarrow{\; i_* \;} \HH{2n-4}X\R \xrightarrow{\; \iota^* \;} \HH{2n-4}Z\R. \]
By \cref{local structure}, $\dim_\R Z \le 2n - 6$, so the outer terms vanish and $i_*$ is an isomorphism.
We now define the second orbifold Chern class $\ccorb2X \in \HH{2n-4}X\R \dual$ to be $(i_*^{-1})\dual \big( \pd(\ccorb2{X^\circ}) \big)$.

\begin{rema}
The intersection of $\ccorb2X$ with a \kahler class can be described more explicitly in terms of differential forms.
For simplicity, assume that $X$ is a compact klt threefold with just a single isolated singularity $p \in X$.
We may assume $p \in X$ to be non-quotient, so that the orbifold locus $X^\circ = X \setminus \set p$ is smooth.
Let $c = \Per(\kappa) \in \HH2X\R$ be a \kahler class, where $\kappa \in \KK X(X)$ is represented by a family $(U_i, \phi_i)_{i \in I}$ of smooth strictly plurisubharmonic functions whose differences $\phi_i - \phi_j$ are pluriharmonic.
Pick an index $\ell \in I$ with $p \in U_\ell$.
Let $\lambda \from X \to [0, 1]$ be a cutoff function near $p$, that is, $\lambda \equiv 1$ in a neighborhood of $p$ and $\supp(\lambda) \subset U_\ell$.
Then $\lambda \cdot \phi_\ell$, extended by zero outside of $U_\ell$, is a smooth function on $X$.

The element $\wt\kappa \in \KK X(X)$ represented by $(U_i, \wt\phi_i)_{i \in I}$, where $\wt\phi_i = \phi_i - \lambda \phi_\ell$, satisfies $\Per(\kappa') = c$.
For each index $i$, we may consider the real $(1, 1)$-form $\i \del\delbar \wt\phi_i$ on the complex manifold $U_i^\circ \coloneqq U_i \setminus \set p$.
Since $\wt\phi_i - \wt\phi_j = \phi_i - \phi_j$ is pluriharmonic, these forms glue to a real $(1, 1)$-form $\omega$ on $X^\circ$.
Furthermore, since $\wt\phi_\ell$ is zero in a neighborhood of $p$, the form $\omega$ has compact support.
Picking a \kahler metric $h$ on $X^\circ$, we obtain the second Chern form $\cc2{X^\circ, h}$, which is a real $(2, 2)$-form on $X^\circ$.
The $6$-form $\cc2{X^\circ, h} \wedge \omega$ has compact support, hence integrates to a finite value.
We then have
\[ \ccorb2X \cdot c = \int_{X^\circ} \cc2{X^\circ, h} \wedge \omega \in \R. \]
\end{rema}

\begin{rema}
\cref{c2orb} can obviously be extended to define the $k$-th orbifold Chern class $\ccorb kX \in \HH{2n-2k}X\R \dual$ of a compact complex space $X$ whose non-orbifold locus has codimension $\ge k + 1$.
However, we do not know any non-trivial natural condition guaranteeing this property for $k \ge 3$.
\end{rema}

\subsection{Proof of \cref{c2bir}}

First we show the existence of a resolution which is minimal in codimension two.
Consider the functorial resolution $f \from Y \to X$, which is projective (see~\cite[Thms.~3.35,~3.45]{Kol07}).
Let $Z \subset X$ be the subset from \cref{local structure}.
Locally at any point $x \in X \setminus Z$, either $X$ is smooth and $f$ is an isomorphism, or $X \cong S \times \C^n$ for a surface $S$ with klt singularities.
In the latter case, we have $Y \cong \wt S \times \C^n$ with $\wt S \to S$ the functorial resolution, since taking the functorial resolution commutes with smooth morphisms in the sense of~\cite[3.34.1]{Kol07}.
But $\wt S \to S$ is the minimal resolution, and then also $f$ is minimal at $x$.

It remains to show well-definedness, i.e.~independence of $\cc2X$ of the resolution chosen.
To this end, suppose that $f \from Y \to X$ and $g \from Z \to X$ are two resolutions minimal in codimension two.
Let $S \subset X$ be an analytic subset of codimension $\ge 3$ such that $Y \setminus f^{-1}(S) \to X \setminus S$ and $Z \setminus g^{-1}(S) \to X \setminus S$ are minimal resolutions.
Consider $W$ the normalization of the main component of the fibre product $Y \x_X Z$ and pick a projective strong resolution $\lambda \from \wt W \to W$ of $W$.
We then have the following commutative diagram:
\[ \xymatrix{
& \wt W \ar^\lambda[d] \ar@/_1.5ex/_{\wt p}[ddl] \ar@/^1.5ex/^{\wt q}[ddr] & \\
& W \ar_-p[dl] \ar^-q[dr] \ar^-r[dd] & \\
Y \ar_-f[dr] & & Z \ar^-g[dl] \\
& X. &
} \]
Furthermore we set $\wt r \coloneqq r \circ \lambda \from \wt W \to X$.

\begin{clai} \label{728}
Let $E_0 \subset \wt W$ be a prime divisor with $\wt r(E_0) \not\subset S$.
Then $E_0$ is $\wt p$-exceptional if and only if it is $\wt q$-exceptional.
\end{clai}

\begin{proof}[Proof of~\cref{728}]
This is well-known, but we recall the proof.
Disregarding $S$ and its respective preimages, we may assume that $f$ and $g$ are minimal.
Write as usual
\begin{equation} \label{738}
K_{\wt W} = \tilde p^* K_Y + E = \tilde q^* K_Z + F,
\end{equation}
where $E$ is effective and $\wt p$-exceptional with support equal to $\Exc(\wt p)$, and likewise for $F$.
Arguing by contradiction, assume that there is a $\wt p$-exceptional prime divisor $E_0 \subset \wt W$ that is not $\wt q$-exceptional.
Set
\[ E' \coloneqq E - \min \{ E, F \}, \quad F' \coloneqq F - \min \{ E, F \}, \]
where the minimum is taken coefficient-wise.
Then $E'$ and $F'$ are effective with no common components.
Furthermore $E' \ne 0$ since $E_0 \subset \supp(E')$.
Since $\wt p$ is a projective morphism, the Negativity Lemma implies that some component of $E'$ is covered by $\wt p$-exceptional curves $C$ satisfying $E' \cdot C < 0$.
For a general such curve, $F' \cdot C \ge 0$ since $C$ is not contained in $F'$.
Now by (\ref{738}),
\[ \underbrace{\big( \tilde p^* K_Y + E' \big) \cdot C}_{=E'\cdot C<0} = (\tilde q^* K_Z + F') \cdot C = \underbrace{K_Z \cdot \tilde q_* C}_{\ge 0} + \underbrace{F' \cdot C}_{\ge 0}. \]
Here the first summand on the right-hand side is non-negative since $K_Z$ is $g$-nef and $\tilde q_* C$ is $g$-exceptional (or zero).
This is the desired contradiction.
If instead there is a $\wt q$-exceptional prime divisor that is not $\wt p$-exceptional, the argument is similar.
\end{proof}

\begin{clai} \label{757}
We have $\mathrm{codim}_X \big( \wt r \big( \!\Exc(\wt p) \cup \Exc(\wt q) \big) \big) \ge 3$.
\end{clai}

\begin{proof}[Proof of~\cref{757}]
We will only show $\mathrm{codim}_X \big( \wt r \big( \!\Exc(\wt p) \big) \big) \ge 3$, since the argument for $\wt q$ is similar.
Since $Y$ is smooth and $\lambda$ is a strong resolution, we have $\lambda \big( \!\Exc(\lambda) \big) \subset W_\sg \subset \Exc(p)$ and hence
\[ \wt r \big( \!\Exc(\wt p) \big) \subset r \big( \lambda \big( \!\Exc(\lambda) \cup \lambda^{-1}(\Exc(p)) \big) \big) \subset r \big( \!\Exc(p) \big). \]
Thus it suffices to show $\mathrm{codim}_X \big( r(B) \big) \ge 3$ for any irreducible component $B \subset \Exc(p)$.
We may assume that $B$ is a divisor and that $r(B) \not\subset S$, since otherwise the claim is clear.
The divisor $\lambda^{-1}_* (B)$ is $\wt p$-exceptional and hence also $\wt q$-exceptional by \cref{728}.
Thus $B$ is $p$- and $q$-exceptional.
So we have maps $p \from B \to Y$ and $q \from B \to Z$ with $\dim p(B), \dim q(B) \le \dim B - 1$ and the further property that $(p, q) \from B \to Y \x Z$ is finite (by the construction of $W$).
Since $r$ factors through both $p$ and $q$, this easily implies $\dim r(B) \le \dim B - 2 = \dim X - 3$.
\end{proof}

By \cref{757}, we may apply \cref{cc pullback lemma} with $Y_1 = \wt W$ and $U = X \setminus \wt r \big( \!\Exc(\wt p) \big)$ to conclude that for any $a \in \HH{2n-4}X\R$ we have
\[ \int_{\wt W} \cc2{\wt W} \cup \tilde r^*(a) = \int_Y \cc2Y \cup f^*(a). \]
By the same reasoning applied to $Z$ instead of $Y$,
\[ \int_{\wt W} \cc2{\wt W} \cup \tilde r^*(a) = \int_Z \cc2Z \cup g^*(a). \]
This shows that $\cc2X \cdot a = \int_Y \cc2Y \cup f^*(a) = \int_Z \cc2Z \cup g^*(a)$ is well-defined, as desired. \qed

\subsection{Proof of \cref{c2orb/bir behaviour}}

That $X_1$ again has klt singularities follows from \cite[Prop.~5.20]{KM98}. \
(\ref{c2orb cover}) holds since in local uniformizations, the map $g$ becomes \'etale.
For more details, see the proof of~\cite[Lemma~2.7]{LT14}.
For (\ref{c2bir cover}), let $f \from Y \to X$ be a resolution minimal in codimension two and consider the commutative diagram
\[ \xymatrix{
Y_1 \ar^{h}[rr] \ar_{f_1}[d] & & Y \ar^f[d] \\
X_1 \ar^g[rr] & & X,
} \]
where $Y_1$ is a strong resolution of the main component of the fibre product $Y \x_X X_1$.
Then $f_1 \from Y_1 \to X_1$ is minimal in codimension two and we have
\begin{align*}
\cc2{X_1} \cdot g^*(a) & = \int_{Y_1} \cc2{Y_1} \cup f_1^*(g^* a) && \text{by definition} \\
  & = \int_{Y_1} \cc2{Y_1} \cup h^*(f^* a) && \text{since $g \circ f_1 = f \circ h$} \\
  & = \deg(h) \cdot \int_Y \cc2{Y} \cup f^*(a) && \text{by \cref{cc pullback lemma}} \\
  & = \deg(g) \cdot \cc2{X} \cdot a && \text{since $\deg g = \deg h$.}
\end{align*}
This ends the proof. \qed

\section{The projective case}

In this section we prove \cref{main thm I} and \cref{main thm II} in the case where $X$ is assumed to be projective.

\begin{theo} \label{main thm proj}
Let $X$ be a projective threefold with canonical singularities and $\cc1X = 0$.
Assume that $\ccorb2X \cdot \omega = 0$ for some \kahler class $\omega \in \HH2X\R$.
Then there exists an abelian threefold $T$ and a finite group $G$ acting on $T$ holomorphically and freely in codimension one such that $X \isom \factor TG$.
\end{theo}

\begin{coro} \label{main cor proj}
Let $X$ be as above, but assume that $\cc2X \cdot \omega = 0$ for some \kahler class $\omega$ on $X$.
Then $X \isom \factor TG$ as above, where $G$ acts freely in codimension \emph{two}.
\end{coro}

The following proposition is crucial to the proof of \cref{main thm proj}.

\begin{prop}[Algebraicity of Chern classes] \label{c2 alg}
Let $X$ be an $n$-dimensional normal projective variety with klt singularities.
Then under the isomorphism
\[ \HH{2n-4}X\R \dual \isom \Hh{2n-4}X\R \]
from the universal coefficient theorem, $\ccorb2X$ is mapped to an element of $\Bb{2n-4}X$ (see~\ref{Bb2kX}).
\end{prop}

Unfortunately, the proof of \cref{c2 alg} is slightly involved.
Essentially, it consists in comparing our definition of $\ccorb2X$ for complex spaces to the algebraic definition for quasi-projective \Q-varieties given by Mumford~\cite[Part~I]{Mumford83}, denoted here by $\ccmum2X$.
We will freely use notation from~\cite{Mumford83} and from~\cite[Section~3]{GKPT16}, to which we refer the reader.

\begin{defi}[Cycle class map for orbifolds] \label{cycle class}
Let $U$ be a quasi-projective variety with quotient singularities, of pure dimension $n$.
For any integer $k \ge 0$, we define the cycle class map
\[ [-] \from \Aa kU \to \HH{2(n-k)}U\R \]
from the Chow group of $U$ to cohomology by sending an algebraic $k$-cycle $Z$ on $U$ first to its fundamental class $[Z] \in \HBM{2k}U$ (see \ref{item class}) and then using the composition
\[ \HBM{2k}U \lto \Hom_\Z \big( \HHc{2k}U\Z, \Z \big) \xrightarrow{-\tensor \R\,} \HHc{2k}U\R \dual \xrightarrow{\pd^{-1}} \HH{2(n-k)}U\R, \]
where the first map is~\cite[IX.1.7]{Iversen86}.
The de Rham interpretation of this map is that of integrating compactly supported $2k$-forms on $U$ over $Z$.
\end{defi}

\begin{proof}[Proof of \cref{c2 alg}]
Let $X^\circ \subset X$ be the orbifold locus of $X$.
Then $X^\circ$ can be equipped with the structure of a \Q-variety.
Let $p \from \wt{X^\circ} \to X^\circ$ be a global cover.
By construction, this comes with a finite group action $G \acts \wt{X^\circ}$ such that $X^\circ = \factor{\wt{X^\circ}}G$ and $p$ is the quotient map.
Moreover let $\pi \from Z \to \wt{X^\circ}$ be a resolution of singularities.
We obtain the following commutative diagram.
\[ \xymatrix{
Z \ar^\pi[r] \ar_\alpha[dr] & \wt{X^\circ} \ar^p[d] & \\
& X^\circ \ar@{ ir->}[r] & X.
} \]
The tangent sheaf $\sF \coloneqq \sT_{X^\circ}$ of $X^\circ$ naturally is a \Q-sheaf, hence it gives rise to a coherent $G$-sheaf $\wt\sF$ on $\wt{X^\circ}$, as described in~\cite[p.~277]{Mumford83}.
Since $\sT_{X^\circ}$ even is a \Q-bundle~\cite[Construction~3.8]{GKPT16}, $\wt\sF$ will be locally free.
Mumford now defines the Chern classes $\cc k{\wt\sF} \in \opA k{\wt{X^\circ}}$ by writing some twist of $\wt\sF$ as a quotient of a trivial vector bundle and pulling back the Chern classes of the tautological quotient bundle on a suitable Grassmannian.
The algebraic orbifold Chern class $\ccmum k{X^\circ} \in \Aa{n-k}{X^\circ}$ is, by definition, the image of $\frac1{\deg p} \cdot \cc k{\wt\sF}$ under the isomorphism~\cite[Theorem~3.1]{Mumford83}
\[ \Aa{n-k}{X^\circ} \isom \opA k{\wt{X^\circ}}^G. \]

\begin{clai} \label{963}
For each $1 \le k \le n$, we have $\ccorb k{X^\circ} = [\ccmum k{X^\circ}] \in \HH{2k}{X^\circ}\R$, where $[-]$ denotes the cycle class map as in \cref{cycle class}.
\end{clai}

\begin{proof}[Proof of \cref{963}]
Consider the vector bundle $\sF_Z \coloneqq \pi^* \wt\sF$.
Since $Z$ is smooth, so is the associated projective bundle $\rho \from \P(\sF_Z) \to Z$.
Let $\hat\xi \coloneqq \ccmum1{\O{\P(\sF_Z)}(1)} \in \AA1{\P(\sF_Z)}$.
Then we have the well-known relation
\begin{equation} \label{Gro rel}
\hat\xi^n = \sum_{i=1}^n (-1)^{i-1} \ccmum i{\sF_Z} \cdot \hat\xi^{n-i} \in \AA n{\P(\sF_Z)},
\end{equation}
since this already holds on the Grassmannian.
The cycle class map $[-] \from \AA*{\P(\sF_Z)} \to \HH{2*}{\P(\sF_Z)}\R$ satisfies $[\hat\xi] = \xi \coloneqq \cc1{\O{\P(\sF_Z)}(1)} \in \HH2{\P(\sF_Z)}\R$.
Hence applying it to \eqref{Gro rel} yields
\[ \xi^n = \sum_{i=1}^n (-1)^{i-1} [\ccmum i{\sF_Z}] \cdot \xi^{n-i} \in \HH {2n}{\P(\sF_Z)}\R. \]
The $\cc i{\sF_Z} \in \HH{2i}Z\R$ satisfy the same relation, which can be seen for example from the short exact sequence
\[ 0 \lto \underbrace{\ker(\beta)}_{\mathrm{c}_n = 0} \lto \rho^* \sF_Z \tensor \O{\P(\sF_Z)}(-1) \xrightarrow{\;\beta\;} \O{\P(\sF_Z)} \lto 0. \]
By the Leray--Hirsch theorem, $\HH*{\P(\sF_Z)}\R$ is a free module over $\HH*Z\R$, generated by $1, \xi, \dots, \xi^{n-1}$.
It follows that $\cc i{\sF_Z} = [\ccmum i{\sF_Z}]$.
But $\cc i{\sF_Z} = \alpha^* \ccorb i{X^\circ}$ and $[\ccmum i{\sF_Z}] = \alpha^* [\ccmum i{X^\circ}]$.
We conclude by injectivity of $\alpha^*$ (see \cref{coho inj} below).
\end{proof}

Pick an algebraic $(n-2)$-cycle $\gamma = \sum n_i \gamma_i$ on $X^\circ$ (with rational coefficients) which represents the algebraic second Chern class $\ccmum2{X^\circ} \in \Aa{n-2}{X^\circ}$, and let $\wb\gamma \coloneqq \sum n_i \wb{\gamma_i} \in \Aa{n-2}X$ be its Zariski closure.
By definition, its fundamental class $\fund{\wb\gamma} \in \Hh{2n-4}X\R$ is contained in $\Bb{2n-4}X$.
We claim that the image of $\ccorb2X$ in $\Hh{2n-4}X\R$ is exactly $\fund{\wb\gamma}$, which will finish the proof of \cref{c2 alg}.

Equivalently, the image of $\fund{\wb\gamma}$ in $\HH{2n-4}X\R \dual$ is $\ccorb2X$.
Consider the chain of isomorphisms
\[ \HH{2n-4}X\R \dual \xrightarrow{(i_*)\dual} \HHc{2n-4}{X^\circ}\R \dual \xrightarrow{\pd^{-1}} \HH4{X^\circ}\R, \]
where $i \from X^\circ \inj X$ is the inclusion and $\pd$ denotes the Poincar\'e duality map (cf.~\cref{c2orb explain} and~\cref{pd orb}).
By definition, the image of $\ccorb2X$ on the right-hand side is $\ccorb2{X^\circ}$.
Since $\wb\gamma \cap X^\circ = \gamma$, it is also clear that the image of $\fund{\wb\gamma}$ on the right-hand side is the cycle class of $\gamma$, which in turn equals $[\ccmum2{X^\circ}]$.
These two elements agree by \cref{963}.
\end{proof}

\begin{lemm} \label{coho inj}
Let $U, \wt U$ be (not necessarily compact) complex spaces with quotient singularities and $f \from \wt U \to U$ a proper, surjective and generically finite map.
Then for any $k \ge 0$, the map $f^* \from \HH kU\R \to \HH k{\wt U}\R$ is injective.
\end{lemm}

\begin{proof}
Let $n \coloneqq \dim U = \dim \wt U$.
We define the Gysin map $f_* \from \HH k{\wt U}\R \to \HH kU\R$ as the composition
\[ \HH k{\wt U}\R \xrightarrow{\pd} \HHc{2n-k}{\wt U}\R \dual \xrightarrow{(f^*)\dual} \HHc{2n-k}U\R \dual \xrightarrow{\pd^{-1}} \HH kU\R. \]
One checks that $f_* \circ f^* = (\deg f) \cdot \id$, where $\deg(f) \ne 0$ is the cardinality of a general fibre of $f$.
Hence $f^*$ is injective, as claimed.
\end{proof}

\begin{rema}
\cref{coho inj} continues to hold even if $f$ is not generically finite, as long as $U$ is compact and $\wt U$ is \kahler.
However, we will not need this.
\end{rema}

\begin{nota}
Let $X$ be as in \cref{main thm proj} and let $H$ be an ample line bundle on $X$.
Then we will write $\ccorb2X \cdot H$ as a shorthand for $\ccorb2X \cdot \cc1{\O X(H)}$.
This is compatible with the notation introduced in~\cite[Introduction]{SBW94}, cf.~\cite[Theorem~3.13.2]{GKPT16}.
\end{nota}

\begin{prop}[Miyaoka semipositivity] \label{miyaoka}
Let $X$ be a projective threefold with canonical singularities.
Assume that $K_X$ is nef.
Then for all ample line bundles $H$ on $X$ we have
\[ 0 \le \ccorb2X \cdot H \stackrel{(*)}\le \cc2X \cdot H, \]
with equality in $(*)$ if and only if $X$ is smooth in codimension two.
\end{prop}

\begin{proof}
The first inequality is~\cite[Theorem~4.2(ii)]{SBW94}.
If $X$ is smooth in codimension two, then $(*)$ is an equality by \cref{c2 comparison}.
If $X$ does have singularities in codimension two, then $\ccorb2X \cdot H < \cc2X \cdot H$ by~\cite[Proposition~1.1]{SBW94}.
\end{proof}

\begin{prop}[Criterion for vanishing of $\mathrm{\tilde c}_2$] \label{c_2=0 crit}
Let $X$ be a projective threefold with canonical singularities and $K_X$ nef.
The following are equivalent:
\begin{enumerate}
\item\label{c2.1} There exists a \kahler class $\omega \in \HH2X\R$ such that $\ccorb2X \cdot \omega = 0$.
\item\label{c2.3} There exists an $\R$-ample class $h \in \NN1X$ such that $\ccorb2X \cdot h = 0$.
\item\label{c2.2} We have $\ccorb2X = 0 \in \HH2X\R \dual$.
\end{enumerate}
\end{prop}

\begin{proof}
\ref{c2.2} trivially implies \ref{c2.1}.
So assume \ref{c2.1}.
By \cref{decomp}, there exist elements $h \in \NN1X$ and $t \in \TR X$ such that $\omega = h + t$ in $\HH2X\R$.
In view of \cref{c2 alg}, we have $\ccorb2X \cdot t = 0$ and hence $\ccorb2X \cdot h = \ccorb2X \cdot \omega = 0$.
By \cref{kahler kleiman} we know that $h$ is $\R$-ample.
This yields \ref{c2.3}.

Assume \ref{c2.3} now.
Let $b \in \HH2X\R$ be arbitrary, and let $b = d + s$ be the decomposition according to \cref{decomp}.
We aim to show that $\ccorb2X \cdot b = 0$, which by \cref{c2 alg} again is equivalent to $\ccorb2X \cdot d = 0$.
Arguing by contradiction, assume first that $\ccorb2X \cdot d > 0$.
Since $h$ is $\R$-ample, there exists a number $\epsilon > 0$ such that $h - \epsilon d$ is still $\R$-ample.
Moreover $\ccorb2X \cdot (h - \eps d) < 0$.
Perturbing $h - \eps d$ slightly and clearing denominators, we find an ample Cartier divisor class $h'$ such that $\ccorb2X \cdot h' < 0$.
This however contradicts \cref{miyaoka}.
In case $\ccorb2X \cdot d < 0$ we argue similarly.
\ref{c2.2} follows.
\end{proof}

\begin{proof}[Proof of \cref{main thm proj}]
Since $\cc1X = 0$, in particular $K_X$ is nef.
By means of \cref{c_2=0 crit}, we deduce that $\ccorb2X = 0$.
The claim now follows from \cite[Main Theorem, p.~266]{SBW94}.
\end{proof}

\begin{proof}[Proof of \cref{main cor proj}]
As above, write $\omega = h + t$ according to \cref{decomp}.
By \cref{pullback T}, we deduce that $\cc2X \cdot h = 0$.
\cref{miyaoka} implies that also $\ccorb2X \cdot h$ vanishes, hence by \cref{main thm proj} we can write $X \isom \factor TG$, where $T$ is an abelian threefold, $G = \Gal(T/X)$ is a finite group and the quotient map $T \to X$ is quasi-\'etale.
Now $X$ is smooth in codimension two by the second part of \cref{miyaoka}, so $T \to X$ is \'etale in codimension two by purity of branch locus.
This implies that the action $G \acts T$ is free in codimension two.
\end{proof}

\begin{rema} \label{higher-dim}
The above approach for proving \cref{main thm proj} does not work in higher dimensions.
To be more precise, suppose we are in the setting of \cref{c_2=0 crit}, but $X$ has dimension $4$.
Given a \kahler class $\omega$ on $X$ such that $\ccorb2X \cdot \omega^2 = 0$ and writing $\omega = h + t$ according to \cref{decomp}, we would like to have $\ccorb2X \cdot h^2 = 0$.
Writing $\omega^2 = h^2 + 2 h \cup t + t^2$ and observing that the middle term integrates to zero because $\ccorb2X \cup h \in \Bb2X$, we are led to the following question:
\emph{Let $Y$ be a complex projective variety with canonical singularities and $\sigma \in \TR Y \subset \HH2Y\R$.
Do we have $\int_S \sigma \cup \sigma = 0$ for any algebraic surface $S \subset Y$?}
The (easy) example below shows that this fails even if $Y$ is a manifold.
\end{rema}

\begin{exam}
Let $E$ be a sufficiently general elliptic curve (i.e.~without complex multiplication), and set $Y = E \x E$ with projections $p, q \from Y \to E$.
Pick two linearly independent classes $\alpha, \beta \in \HH1E\R$, and consider $\sigma \coloneqq p^* \alpha \cup q^* \alpha + p^* \beta \cup q^* \beta$.
Then $\sigma \cup \sigma = -2 \cdot p^*(\alpha \cup \beta) \cup q^*(\alpha \cup \beta) \ne 0 \in \HH4Y\R$, but $\sigma$ is zero on the fibres of $p$ and $q$ as well as on the diagonal of $Y$.
It follows that $\sigma \in \TR Y$ since these classes generate $\Bb2X$.
\end{exam}

\section{The case of nonzero irregularity}

In this section we prove \cref{main thm I} and \cref{main thm II} in case $X$ has trivial canonical bundle and non-trivial Albanese torus.

\begin{theo} \label{main res irreg}
Let $X$ be a compact complex threefold with canonical singularities.
Assume that $\Can X \isom \O X$ and $q(X) \coloneqq \dim_\C \HH1X{\O X} > 0$.
\begin{enumerate}
\item\label{main thm irreg} If $\ccorb2X \cdot \omega = 0$ for some \kahler class $\omega \in \HH2X\R$, then there exists a $3$-dimensional complex torus $T$ and a finite group $G$ acting on $T$ holomorphically and freely in codimension one such that $X \isom \factor TG$.
\item\label{main cor irreg} If $\cc2X \cdot \omega = 0$ for some \kahler class $\omega$ on $X$, then $X \isom \factor TG$ as before, where $G$ acts freely.
In particular, $X$ is smooth.
\end{enumerate}
\end{theo}

\subsection{Uniformization in dimension two}

The idea of the proof is, in a sense, to use the Albanese map of $X$ as a replacement for cutting down by hyperplane sections.
The following proposition is then applied to the fibres of $\alb_X$.

\begin{prop}[Torus quotients in dimension two] \label{surface uniformization}
Let $S$ be a compact \kahler surface with klt singularities satisfying $\cc1S = 0$ and $\ccorb2S = 0$.
Then there exists a $2$-dimensional complex torus $T$ and a finite group $G$ acting on $T$ holomorphically and freely in codimension one such that $S \isom \factor TG$.
\end{prop}

The proof of \cref{surface uniformization} relies on the following statement about \'etale fundamental groups of klt surfaces.
Here, we \emph{define} the \'etale fundamental group $\piet X$ of a complex space $X$ to be the profinite completion of its topological fundamental group $\pi_1(X)$.
This is compatible with the standard usage in algebraic geometry~\cite[Fact~1.6]{GKP13}.

\begin{prop}[\'Etale fundamental groups of surfaces] \label{piet}
Let $S$ be a compact complex surface with klt singularities.
Then there exists a finite quasi-\'etale Galois cover $\gamma \from T \to S$, with $T$ normal (hence klt), such that the map $\piet{T\sm} \to \piet T$ induced by the inclusion of the smooth locus is an isomorphism.
\end{prop}

This result has been proven in~\cite[Theorem~1.5]{GKP13} for quasi-projective klt varieties of any dimension.
In~\cite{BCGST16}, the first author together with his coauthors generalized it to positive characteristic, namely to $F$-finite Noetherian integral strongly $F$-regular schemes.
It would be equally interesting to consider the case of arbitrary (compact) complex spaces with klt singularities.
However, this is not needed for our present purposes.

\begin{proof}[Proof of \cref{piet}]
We follow the proof of~\cite[Theorem~1.5]{GKP13}.
Assume that the desired cover does not exist.
Then for every finite quasi-\'etale Galois cover $\wt S \to S$ there exists a further cover $\wh S \to \wt S$ which is quasi-\'etale but not \'etale.
Iterating this argument and taking Galois closure, one obtains a sequence of covers
\[ \cdots \xrightarrow{\;\gamma_3\;} S_2 \xrightarrow{\;\gamma_2\;} S_1 \xrightarrow{\;\gamma_1\;} S_0 = S \]
such that each $\gamma_i$ is quasi-\'etale but not \'etale, and each $\delta_i \coloneqq \gamma_1 \circ \cdots \circ \gamma_i$ is Galois.
For every index $i$, there exists a (necessarily singular) point $p_i \in S$ such that $\gamma_i$ is not \'etale over some point of $\delta_{i-1}^{-1}(p_i)$.
Since the singular set $S_\sg$ is finite, we have $p_i = p_0$ (say) for infinitely many $i$.

By \cref{local structure}, we may choose a sufficiently small neighborhood $p_0 \in U_0 \subset S$ admitting a local uniformization $(V_0, G, \phi)$, where $G$ acts on $0 \in V_0 \subset \C^2$ freely outside of the origin.
Set $U_0^\x \coloneqq U_0 \setminus \set{p_0}$.
Shrinking $U_0$ if necessary, we may assume that each connected component of each $\delta_i^{-1}(U_0)$ contains exactly one point mapping to $p_0$.
Choose a sequence of connected components $W_i \subset \delta_i^{-1}(U_0)$ such that $W_i \subset \gamma_i^{-1}(W_{i-1})$.
Let $t_i \in W_i$ be the unique point mapping to $p_0$, and set $W_i^\x \coloneqq W_i \setminus \set{t_i}$.
Then
\[ G_i \coloneqq \delta_{i*} \big( \pi_1(W_i^\x) \big) \subset \pi_1(U_0^\x) \isom G \]
defines a decreasing sequence of (normal) subgroups of $G$.
Whenever $W_i \xrightarrow{\gamma_i} W_{i-1}$ is not \'etale, $W_i^\x \xrightarrow{\gamma_i} W_{i-1}^\x$ has degree $\ge 2$ and consequently $G_i \subsetneq G_{i-1}$.
As this happens for infinitely many indices $i$, the sequence $(G_i)$ does not stabilize.
This is impossible because $G$ is finite.
\end{proof}

\begin{proof}[Proof of \cref{surface uniformization}]
Since $S$ has only quotient singularities (\cref{local structure}), it is an \emph{orbifolde pure} in the sense of~\cite[D\'efinition~3.1]{Cam04FanoConference}.
By~\cite[Th\'eor\`eme~4.1]{Cam04FanoConference}, $S$ carries a Ricci-flat orbifold \kahler metric $g$.
This means that the tangent \Q-bundle $(\sT_S, g)$ is Hermite--Einstein over $(S, g)$.
By~\cite[Theorem~4.4.11]{KobayashiDGCVB}, $\ccorb2S = 0$ implies that the Chern connection on $(\sT_S, g)$ is flat\footnote{
The cited reference only treats vector bundles over \kahler manifolds. But note that the proof consists of purely local calculations, which in our situation can still be done in local uniformizations $(\wt U_\alpha, G_\alpha, \phi_\alpha)$ of open sets $U_\alpha \subset S$ covering $S$. Therefore the result continues to hold for complex spaces $S$ with at worst quotient singularities.}.
In particular, the tangent bundle of the smooth locus $\sT_{S\sm}$ is given by a linear representation $\rho \from \pi_1(S\sm) \to \GL2\C$.
The finitely generated group $G \coloneqq \rho \big( \pi_1(S\sm) \big) \subset \GL2\C$ is residually finite, meaning that the natural map to the profinite completion $G \to \wh G$ is injective~\cite{Platonov68}.
Let $\gamma \from T \to S$ be the cover given by \cref{piet}, and set $T^\circ \coloneqq \gamma^{-1}(S\sm)$.
Then the tangent bundle $\sT_{T^\circ} = \gamma^* \sT_{S\sm}$ is given by $\rho^\circ \from \pi_1(T^\circ) \to \GL2\C$, the pullback of $\rho$.
We have a commutative diagram
\[ \xymatrix{
\wh G & & \piet{T^\circ} \ar_-{\wh{\rho^\circ}}[ll] \ar^-\sim[r] & \piet{T\sm} \ar^-\sim[r] & \piet T \\
G \ar@{ ir->}[u] & & \pi_1(T^\circ) \ar_-{\rho^\circ}[ll] \ar[u] \ar@{->>}[rr] & & \pi_1(T). \ar[u]
} \]
Hence $\rho^\circ$ factorizes via a representation $\wb\rho \from \pi_1(T) \to G$.
By construction, on $T^\circ$ the associated locally free sheaf $\sF_{\wb\rho}$ agrees with the tangent sheaf $\sT_T$.
As both sheaves are reflexive and $\codim{T \setminus T^\circ}T \ge 2$, they are in fact isomorphic.
In particular, $\sT_T$ is locally free.
By the known cases of the Lipman--Zariski conjecture\footnote{
The cited references only consider algebraic varieties. However, $T$ has at worst quotient singularities and these are automatically algebraic.
}, $T$ is smooth~\cite[Theorem~1.1]{Dru13}, \cite[Corollary~1.3]{GK13}.
Now classical differential geometry implies that $T$ is the quotient of a complex torus by a finite group acting freely~\cite[Corollary~4.4.15]{KobayashiDGCVB}.
We conclude by \cref{galois} below.
\end{proof}

The following well-known argument will be used several times in the sequel.

\begin{lemm}[Galois closure trick] \label{galois}
Let $X, Y$ be normal compact complex spaces, and let $\gamma \from Y \to X$ be a finite surjective map \'etale in codimension $k \ge 1$.
If $Y$ is the quotient of a complex torus $T$ by a finite group acting freely in codimension $k$, then the same is true of $X$ (for a complex torus $T'$ isogenous to $T$).
\end{lemm}

\begin{proof}
Let $T \to Y$ be the quotient map and consider the Galois closure of the composition $q \from T \xrightarrow{} Y \xrightarrow{\gamma} X$,
\[ \xymatrix{
T' \ar[r] \ar@/^2.5ex/^{\text{Galois}}[rr] & T \ar_q[r] & X.
} \]
Since $q$ is \'etale in codimension $k \ge 1$, so is $T' \to T$.
By purity of branch locus, $T' \to T$ is \'etale and hence $T'$ is a complex torus, isogenous to $T$.
The Galois group $G' \coloneqq \Gal(T'/X)$ acts on $T'$ freely in codimension $k$.
Thus $X = \factor{T'}{G'}$ is a torus quotient as desired.
See~\cite[proof of Corollary~1.16]{GKP13} for more details.
\end{proof}

\subsection{Proof of \cref{main res irreg}}

By \cite[Theorem~1.13]{AlgApprox}, the Albanese map $\alpha\colon X \to A = \Alb(X)$ is a surjective analytic fibre bundle with connected fibre.
Hence $q(X) \in \set{1, 2, 3}$. We make a case distinction according to the value of $q(X)$.

\begin{itemize}
\item \emph{Case~1: $q(X) = 3$.} In this case, $\alpha$ is an isomorphism.
Thus we may take $T \coloneqq X$ and $G \coloneqq \set\id$.

\item \emph{Case~2: $q(X) = 2$.} We have $q(A) = q(X) = 2$ and $q(F) = 1$ for any fibre $F$ of $\alpha$, since all the fibres are isomorphic to a (fixed) elliptic curve.
By \cite[Proposition~4.5]{Fuj83} (see also \cite[Proposition~3.2]{AlgApprox}) there is a finite \'etale cover $\wt A \to A$ such that $\wt X \coloneqq X \x_A \wt A$ has irregularity $q(\wt X) = q(F) + q(A) = 3$.
By Case~1, $\wt X \eqqcolon T$ is a torus.
It now follows from \cref{galois} that $X$ is a torus quotient by a free action.

\item \emph{Case~3: $q(X) = 1$.} The fibre $F$ is a compact \kahler surface with canonical singularities and $\Can F \isom \O F$.
If $\wt F \to F$ is the minimal resolution, then $\wt F$ is a smooth compact \kahler surface with trivial canonical sheaf.
By the Kodaira--Enriques classification of surfaces \cite[Table~10 on p.~244]{BHPV04}, $\wt F$ is either a two-dimensional complex torus or a K3 surface.
If $\wt F$ is a torus, then it does not contain any rational curves and so $F = \wt F$.
Hence $q(F) = 2$ and $q(A) = 1$. We deduce that $X$ is a torus quotient by the same argument as in Case 2.

Note that up to here, we have not used the second Chern class assumption and in particular we have proved \ref{main thm irreg} and \ref{main cor irreg} at the same time.
As there is only one case left, from now on we make the following additional assumption, without loss of generality.

\begin{awlog}
The minimal resolution $\wt F$ of the fibre of the Albanese map is a K3 surface.
\end{awlog}

Using the natural homomorphism $\Aut(F) \to \Aut(\wt F)$ we can construct a fibre bundle $Y \to A$ with fibre $\wt F$ which at the same time is a minimal resolution $f\colon Y \to X$ of $X$ over $A$.
Alternatively, $Y$ can also be constructed as the functorial resolution of $X$ in the sense of~\cite[Thms.~3.35,~3.45]{Kol07}.
Since $\Aut(\wt F)$ is discrete, the transition functions describing the bundle $Y \to A$ are automatically locally constant.
Hence there is a well-defined notion of flat section of $Y \to A$, that is, $Y \to A$ arises from a representation $\rho\from \pi_1(A) \to \Aut(\wt F)$ of the fundamental group of $A$.
Picking a \kahler form $\omega$ on $Y$, the representation $\rho$ factors through $\Aut(\wt F, [\omega|_{\wt F}])$, the group of automorphisms of $\wt F$ fixing the cohomology class of $\omega|_{\wt F}$.
By~\cite[Theorem~4.8]{Fuj78}, the latter group is finite.
The kernel of $\rho$ is therefore a finite index subgroup of $\pi_1(A)$.
It defines a finite \'etale cover $A_1 \to A$ such that $Y_1 \coloneqq Y \x_A A_1$ is isomorphic to $\wt F \x A_1$ over $A_1$.
Then also $X_1 \coloneqq X \x_A A_1$ is isomorphic to $F \x A_1$ over $A_1$.
Writing $p \from X_1 \to F$ and $q \from Y_1 \to \wt F$ for the projections onto the first factor, we thus have the following commutative diagram:
\[ \xymatrix{
\wt F \ar[d] & Y_1 \ar^-h[rr] \ar_-{f_1}[d] \ar_-q[l] & & Y \ar^-f[d] \\
F & X_1 \ar^-g[rr] \ar[d] \ar_-p[l] & & X \ar^-\alpha[d] \\
& A_1 \ar^-{\text{finite \'etale}}[rr] & & A.
} \]
We will compare the second Chern class of $X$ to that of $F$.

\begin{clai} \label{c2 fibre}
In case \ref{main thm irreg}, we have $\ccorb2F = 0$.
In case \ref{main cor irreg}, we have $\cc2F = 0$.
\end{clai}

\begin{proof}
For the first statement, note that $X_1 \isom F \x A_1$ has at worst quotient singularities, hence we may use Poincar\'e duality (\cref{pd orb}) to define the Gysin map $p_* \from \HH*{X_1}\R \to \HH{*-2}F\R$ as the composition
\[ \HH*{X_1}\R \xrightarrow{\pd} \HH{6-*}{X_1}\R \dual \xrightarrow{(p^*)\dual} \HH{6-*}F\R \dual \xrightarrow{\pd^{-1}} \HH{*-2}F\R. \]
By the projection formula we have
\begin{equation} \label{1164}
p_* \big( p^* \ccorb2F \cup g^* \omega \big) = \ccorb2F \cup p_*(g^* \omega) \in \HH4F\R \isom \R.
\end{equation}
According to \eqref{c2orb cover}, it holds that $\ccorb2{X_1} \cdot g^*(\omega) = (\deg g) \cdot \ccorb2X \cdot \omega = 0$.
Since $\ccorb2{X_1} = p^* \ccorb2F$, the left-hand side of (\ref{1164}) is zero.
On the other hand, $p_*(g^* \omega) \ne 0$ since $g^* \omega$ restricted to the fibres of $p$ is a \kahler class by \cref{kahler finite pullback}.
Thus (\ref{1164}) shows that $\ccorb2F = 0$, as desired.

The proof of the second statement is similar, arguing on $Y_1$ instead of $X_1$.
The details are omitted.
\end{proof}

\emph{Proof of \ref{main thm irreg}:}
By \cref{c2 fibre}, $\ccorb2F = 0$ and then by \cref{surface uniformization}, $F = \factor TG$ is the quotient of a complex torus $T$ by a finite group $G$ acting freely in codimension one.
Letting $G$ act trivially on $A_1$, the same is true of $X_1 = \factor{T \x A_1}G$.
Now \ref{main thm irreg} follows from \cref{galois}.

\emph{Proof of \ref{main cor irreg}:}
By \cref{c2 fibre}, $\cc2F = 0$.
But $\cc2F = \cc2{\wt F} = 24$, since $\wt F$ is a K3 surface.
This is a contradiction and shows that this case cannot occur.
We have shown that in each case that actually can occur, $X$ is the quotient of a torus by a free group action.
Thus \ref{main cor irreg} is proved.
\end{itemize}

\noindent
This ends the proof of \cref{main res irreg}. \qed

\section{Proof of main results}

In this section, we prove \cref{main thm I} and \cref{main thm II}.

\subsection{Proof of \cref{main thm I}} \label{main thm I proof}

For ``\ref{mainI.2} $\imp$ \ref{mainI.1}'', let $q \colon T \to X = \factor TG$ be the quotient map.
Then the reflexive tensor power $\omega_X^{[n]} \isom \O X$, where $n = |G|$, and in particular $\cc1X = 0$ in $\HH2X\R$.
Since $T$ is \kahler, so is $X$, see~\cite[Ch.~IV, Cor.~1.2]{Var89}.
Let $\omega \in \HH2X\R$ be a \kahler class on $X$.
By (\ref{c2orb cover}) we have
\[ \ccorb2X \cdot \omega = \frac1n \cdot \ccorb2T \cdot q^*(\omega) = 0, \]
as $\ccorb2T = \cc2T = 0$.

For ``\ref{mainI.1} $\imp$ \ref{mainI.2}'', note that $\Can X$ is torsion by abundance (cf.~e.g.~\cite[Prop.~10.2]{AlgApprox}).
Let $g \from X_1 \to X$ be the index one cover of $X$, where $\Can{X_1} \isom \O{X_1}$ and $X_1$ likewise has canonical singularities~\cite[Def.~5.19 and Prop.~5.20]{KM98}.
If $X_1$ is the quotient of a complex torus $T$ by a finite group acting freely in codimension one, then the same is true of $X$ by \cref{galois}.
Note that $\cc1{X_1} = g^*(\cc1X) = 0$ and that (\ref{c2orb cover}) implies
\[ \ccorb2{X_1} \cdot g^*(\omega) = \deg(g) \cdot \ccorb2X \cdot \omega = 0. \]
Since $g^*(\omega)$ is a \kahler class by \cref{kahler finite pullback}, replacing $X$ by $X_1$ we may assume from now on that $\Can X \isom \O X$.
\hypertarget{Serre}{We make a case distinction:}
\begin{itemize}
\item If $X$ is projective, then \ref{mainI.2} follows from \cref{main thm proj}.
\item \label{Serre page} If $X$ is not projective, then $q(X) = h^1(X, \O X) \ne 0$ by the Kodaira embedding theorem and Serre duality (cf.~the argument in the proof of \cite[Theorem~6.1]{AlgApprox}).
Now \ref{mainI.2} follows from \ref{main thm irreg}.
\end{itemize}
This ends the proof of \cref{main thm I}. \qed

\subsection{Proof of \cref{main thm II}}

This proof is completely analogous to the previous one, and thus it is omitted.
All one needs to do is change the references to (\ref{c2orb cover}) to (\ref{c2bir cover}), \cref{main thm proj} to \cref{main cor proj} and \ref{main thm irreg} to \ref{main cor irreg}.

\subsection{On \cref{main conj} in dimension three} \label{main conj 3}

Here we explain how to prove \cref{main conj} in dimension three, assuming the special case of the Abundance Conjecture mentioned in the introduction.
The direction ``\ref{main conj.2} $\imp$ \ref{main conj.1}'' is easily proved in exactly the same way as ``\ref{mainI.2} $\imp$ \ref{mainI.1}''.
For the other direction, $\Can X$ is torsion by abundance.
Let $g \from X_1 \to X$ be the index one cover of $X$.
Then $X_1$ is also klt and since $\Can{X_1}$ is invertible, the singularities of $X_1$ are in fact canonical.
We may therefore apply \cref{main thm I} to $X_1$ and conclude that $X_1$ is a torus quotient.
By \cref{galois}, also $X$ is a torus quotient.


\newcommand{\etalchar}[1]{$^{#1}$}
\providecommand{\bysame}{\leavevmode\hbox to3em{\hrulefill}\thinspace}
\providecommand{\MR}{\relax\ifhmode\unskip\space\fi MR}
\providecommand{\MRhref}[2]{%
  \href{http://www.ams.org/mathscinet-getitem?mr=#1}{#2}
}
\providecommand{\href}[2]{#2}

\end{document}